\documentclass{article}

\usepackage{arxiv}

\usepackage[utf8]{inputenc} 
\usepackage[T1]{fontenc}    
\usepackage{hyperref}       
\usepackage{url}            
\usepackage{booktabs}       
\usepackage{amsfonts}       
\usepackage{nicefrac}       
\usepackage{microtype}      
\usepackage{lipsum}		
\usepackage{graphicx}
\usepackage{natbib}
\usepackage{doi}

\usepackage{tabularx}
\usepackage{booktabs}
\usepackage{makecell}
\usepackage{bm}
\def\mydefb#1{\expandafter\def\csname b#1\endcsname{{\bm{#1}}}}
\def\mydefallb#1{\ifx#1\mydefallb\else\mydefb#1\expandafter\mydefallb\fi}
\mydefallb ABCDEFGHIJKLMNOPQRSTUVWXYZeucsxgvwktijnr\mydefallb
\def\mydefgreek#1{\expandafter\def\csname b#1\endcsname{\text{\boldmath$\mathbf{\csname #1\endcsname}$}}}
\def\mydefallgreek#1{\ifx\mydefallgreek#1\else\mydefgreek{#1}%
	\lowercase{\mydefgreek{#1}}\expandafter\mydefallgreek\fi}
\mydefallgreek {beta}{Gamma}{Delta}{epsilon}{Omega}{etaex}{Theta}{Iota}{Lambda}{kappa}{mu}{nu}{Xi}{phi}{Pi}{rho}{Psi}{Sigma}{ell}{tau}{varphi}{psi}{lambda}{alpha}\mydefallgreek
\usepackage{bbm}
\usepackage{placeins}

\usepackage[english]{babel}
\usepackage{amsthm}
\newtheorem{theorem}{Theorem}
\newtheorem{remark}{Remark}
\newtheorem{lemma}{Lemma}
\newtheorem{assumption}{Assumption}
\newtheorem{definition}{Definition}
\newtheorem{corollary}{Corollary}
\newtheorem{example}{Example}

\newcommand{\thmBoundtrunc}[1]{Under Assumption \ref{boundedtotrew} there is a unique real number $\nu_\sigma(\bs)$ attaining the supremum $$\sup\left\{\nu:\inf_{\tau\in\mathcal{T}_\sigma}\mathbb{E}_\bs[\tilde{Z}_\tau(\nu)]\leq 0 \right\}.$$
	In fact, $\nu_\sigma(\bs)$ is the unique root of $$
	f_\bs^{\sigma}:\nu\mapsto \inf_{\tau\in\mathcal{T}_\sigma}\mathbb{E}_\bs[\tilde{Z}_\tau(\nu)].$$
	The infimum in the condition is attained for some $\tau_{\sigma}(\nu,\bs)\in\mathcal{T}_{\sigma}$.\\
	Furthermore, with $x^+ =\max(x,0)$ for $x\in\mathbb{R}$, it holds that 
	$$0\leq\nu(\bs)-\nu_\sigma(\bs)\leq \frac{\mathbb{E}_\bs\left[\gamma^\sigma \nu(\bS_\sigma)^+\right]}{1 -\mathbb{E}_\bs[\gamma^\sigma]}
	\label{boundtrunc}.$$}

\newcommand{\thmFTCB}{ For $\omega^*\in\mathbb{R}$, initial 
	point $\omega_0\in\mathbb{R}$, and martingale difference sequence \eqref{MartDiffSeq}, let the sequence $(\omega_m)_m$ be defined as
	\begin{equation*}
		\omega_{m+1} = \omega_m - \alpha_m(f_\bs(\omega_m)-f_\bs(\omega^*)+\epsilon_m).
	\end{equation*}
	Assume the step-size sequence is such that $\sup_m \alpha_m\leq 2(R_u-R_\ell)$ almost surely.
	Then for all $m$
	\begin{equation*}
		\mathbb{E}[(\omega^*-\omega_{m+1})^2] \leq \mathbb{E}[\left(1-c\alpha_m\right)^2(\omega^*-\omega_m)^2] +\mathbb{E}[\alpha_m^2]\mathbb{E}[\epsilon_m^2].
\end{equation*}}

\def\Lemmanubound{
	Let $\sup_m\alpha_m\leq M_\alpha<\infty$ almost surely. We have almost surely that for all $m$ $$\nu_m(\bs)\in[R_\ell-M_{\alpha}/2,\; R_u+M_\alpha/2].$$	 }

\def\Thmderiv{If for all $t$ the cumulative distribution functions of $R(\bS_t)$ starting from $\bS_0=\bs$ have finitely many jumps, then for all but finitely many points $\nu$ the function $\tilde{f}_{\bs}$ is differentiable with derivative $$\frac{d}{d\nu}\tilde{f}_{\bs}(\nu) =  \sum_{k=1}^K\mathbb{E}_\bs\left[h^{(k)}_{[k,1],n}\left(U^{(k)}_{[k,1],n}(\nu),\nu\right)\mathbbm{I}(\tilde{V}_{\bs}(\nu)\in [-1/2,1/2])\right],$$ where 
	$\tilde{V}_{\bs}(\nu) = \sum_{k=1}^K \min_{u\in [N]}Z_{[k,1],u, n}^{(k)}(\nu)$, $U^{(k)}_{\bi,n}(\nu)=\underset{u\in [N]}{\arg\min}\;Z_{\bi,u,n}^{(k)}(\nu),$ and
	
	\begin{align*} &h^{(1)}_{\bi,n}(t,\nu) = (1-\gamma^t)/(2(R_u-R_\ell)(1-\gamma^N)),\\&
		h^{(k)}_{\bi,n}(t, \nu)= h_{[\bi,1],n}^{(k-1)}(t,\nu) - \frac{1}{n(\bi,k)}\sum_{j=2}^{n(\bi,k)+1} h_{[\bi,j],n}^{(k-1)}(U^{(k-1)}_{[\bi,j],n}(\nu),\nu)|\{\bS^{[\bi,j]}_{\ixt}=\bS^{\bi}_{\ixt}\}\;\;\;\;(k\geq 2).
	\end{align*}
}

\def\Thmconvgdist{Let $h_{m,[k,1],n}^{(k)},\; U_{m,[k,1],n}^{(k)}$ be independent versions (in $ m$) of $h_{[k,1],n}^{(k)},\; U_{[k,1],n}^{(k)}$, and
	$$h_{m} : \nu\mapsto \sum_{k=1}^{K}h_{m,[k,1],n}^{(k)}(U_{m,[k,1],n}^{(k)}(\nu),\nu).$$
	Let
	$$\alpha_m = \frac{1}{|\sum_{\ell=1}^m h_\ell(\nu_{\ell}(\bs))|}.$$ 
	Let $\mathcal{V}$ be the set of points where $\tilde{f}_\bs$ is differentiable.
	If $\inf_{\nu\in\mathcal{V}}d/d\nu \tilde{f}_\bs(\nu)>0$, then there is a unique point $\tilde{\nu}(\bs)$ such that $\tilde{f}_{\bs}(\tilde{\nu}(\bs))=0$. If the derivative of $f^{(K)}_{\bs,n}$ exists at $\tilde{\nu}(\bs)$ and $\mathbb{E}[V_{\bs}(\tilde{\nu}(\bs))^2]>0$,
	we have
	\begin{equation}\sqrt{m}(\nu_{m}(\bs)-\tilde{\nu}(\bs))\stackrel{d}{\rightarrow} \mathcal{N}\left(0, \frac{\mathbb{E}[V_{\bs}(\tilde{\nu}(\bs))^2]}{\left(\frac{d}{d\nu}f^{(K)}_{\bs,n}(\tilde{\nu}(\bs))\right)^2}\right).\label{eqn:CLT_SA}\end{equation}}	

\usepackage{stmaryrd }
\newcommand{\ixt}{\llbracket 0, t\rrbracket}
\newcommand{\ixT}{\llbracket 0, T\rrbracket}
\newcommand{\ixv}[1]{\llbracket 0, #1\rrbracket}
\usepackage{xcolor}
\usepackage{mathtools}
\DeclarePairedDelimiter\ceil{\lceil}{\rceil}
\DeclarePairedDelimiter\floor{\lfloor}{\rfloor}

\usepackage{accents}

\newcommand{\ubar}[1]{\underaccent{\bar}{#1}}

\title{A Sampling-based Gittins Index Approximation}

\author{Stef Baas \\
	Stochastic Operations Research Group,
	University~of~Twente,
	Enschede, The~Netherlands\\
	\url{s.p.r.baas@utwente.nl} \\
	\And
	Richard J. Boucherie\\
	Stochastic Operations Research Group,
	University~of~Twente,
	Enschede, The~Netherlands\\
	\And 
	Aleida Braaksma \\
	Stochastic Operations Research Group,
	University~of~Twente,
	Enschede, The~Netherlands\\
}

\renewcommand{\shorttitle}{A Sampling-based Gittins Index Approximation}

\hypersetup{
	pdftitle={A Sampling-based Gittins Index Approximation},
	pdfauthor={Stef Baas, Richard Boucherie, Aleida Braaksma},
	pdfkeywords={Stochastic Approximation,  Multi-Armed Bandits,  Optimal Stopping, Bayesian Computation, Markov~Decision~Processes},
}

\begin{document}
	\maketitle
	
	\begin{abstract}
		A sampling-based method is introduced to approximate the Gittins index for a general family of alternative bandit processes. The approximation consists of a truncation of the optimization horizon and support for the immediate rewards, an optimal stopping value approximation, and a stochastic approximation procedure. Finite-time error bounds are given for the three approximations, leading to a procedure to construct a confidence interval for the Gittins index using a finite number of Monte Carlo samples, as well as an \hbox{epsilon-optimal policy} for the Bayesian multi-armed bandit. Proofs are given for almost sure convergence and convergence in distribution for the sampling based Gittins index approximation. In a numerical study, the approximation quality of the proposed method is verified for the Bernoulli bandit and Gaussian bandit with known variance, and the method is shown to significantly outperform Thompson sampling and the Bayesian Upper Confidence Bound algorithms for a novel random effects multi-armed bandit.
	\end{abstract}

	\keywords{Stochastic Approximation \and Multi-Armed Bandits \and Optimal Stopping \and Bayesian Computation \and Markov~Decision~Processes}

	\section{Introduction}
	The family of alternative bandit processes (FABP) is a well established problem in the field of Operations Research (\citet{glazebrook1983optimal}). In short, an FABP is a problem in which the decision maker sequentially chooses one out of a finite collection of independent Markov reward processes (sometimes called arms) to evolve, and the goal is to optimize the total discounted reward.  
	As initially proven in~\citet{gittins1979bandit}, 
	an index value can be determined based on the state of each bandit process, and the optimal policy is to choose the arm with the highest index value at each decision epoch.
	This index value, introduced as the dynamic allocation index, is now referred to as the Gittins index.  As the Gittins index can be calculated separately based on the state of each bandit process, there is a large gain in computational efficiency for finding the optimal policy when compared to methods taking the states of all bandit processes into account (\citet{puterman1990markov}). Next to the family of alternative bandit processes, the Gittins index was also found to be the optimal solution to a number of other problems in operations research such as optimal scheduling and search problems~(\citet{aalto2011properties, boodaghians2022pandora}).

	When, e.g., the total average reward is considered, or the Markov reward processes are no longer independent, the bandit problem becomes a restless bandit problem, where each arm evolves after an arm is chosen, instead of just that specific arm. Optimality of the Gittins index policy is no longer a guarantee in this case. 
	When the condition of indexability is met, however, the policy choosing the largest Whittle index  can provide a well-behaving heuristic for the problem under consideration~(\citet{
		weber1990index, glazebrook2009generalized,	fryer2018two
	}). The Whittle index has a definition similar to the Gittins index, and hence computation methods for the Gittins index work for the Whittle index as well.
	

	
	Families of alternative bandit processes arise in various applications, such as job scheduling in queues, mining (or search) operations, advertisement, and Bayesian multi-armed bandit problems~(\citet{mahajan2008multi, gittins2011multi}). 
	The current paper will focus mainly on multi-armed bandit problems, with an application to clinical trials. Due to an increase in computational power and an increased focus on patient-centric medicine, research on treatment allocation in clinical trials using the Gittins index has gained popularity in the last decade \hbox{(\citet{ villar2015multi, villar2015FLGI, robertson2023response})}.
	In the multi-armed bandit problem, the decision maker is tasked with the choice to sample rewards from one out of a finite collection of unknown distributions at each decision epoch. In the Bayesian setting, each distribution is endowed with a prior. The resulting sequences of posterior distributions and posterior mean outcomes at each decision epoch then result in the separate Markov reward processes of an FABP. The Bayesian multi-armed bandit problem, popularized in~\citet{robbins1952some}, was introduced in~\citet{thompson1933likelihood} along with the approximate solution method now referred to as Thompson sampling, which was hence the first approximate solution method for the Bayesian multi-armed bandit problem (\citet{Slivkins2019}). 
	
	This paper focuses on maximizing the Bayesian total expected discounted reward. 
	When the regret under a multi-armed bandit problem is analyzed in the frequentist framework, many \hbox{index-based} approximate solution methods exist (\citet{bubeck2012regret,kuleshov2014algorithms, lattimore2020bandit}), either based on frequentist or Bayesian approaches. in~\citet{lai1985asymptotically}, frequentist asymptotic lower bounds were found for the number of suboptimal choices made under a broad class of approximate solution methods. In the case of bounded rewards, the asymptotic regret upper bound equals this lower bound  for e.g., approximate solution methods based on an upper confidence bound (UCB) for the expected reward (\citet{ lai1987adaptive,auer2002finite}), proving that these methods are asymptotically optimal in the frequentist framework. 
	Many Bayesian approximate solution methods are asymptotically optimal and were shown to have excellent empirical frequentist performance (\citet{ kaufmann2012bayesian,kaufmann2018bayesian, lattimore2020bandit}). In~\citet{lattimore2016regret} it was shown that, in the frequentist framework, the Gittins index strategy with improper uniform prior for the mean yields an asymptotically optimal approximate solution method for the Gaussian model, outperforming other optimal Bayesian methods such as the Thompson sampling and Bayesian Upper Confidence Bound (Bayes-UCB).
	
	There is a large amount of literature on calculating the Gittins index, starting with the Calibration method introduced in~\citet{gittins1979bandit}. A survey covering most literature up to 2014 on offline and online algorithms to calculate the Gittins index for general countable state FABPs is provided in~\citet{chakravorty2014multi}. In~\citet{yao2006some}, an approximation method for the Gittins index is derived in the Gaussian bandit with unknown mean and known variance. A numerical approximation for this bandit based on quadratic splines is described in~\hbox{\citet{lattimore2016regret}}. In~\citet{edwards2019practical} the lack of general open source code for calculating the Gittins index is acknowledged, and methods and accompanying code, based on the Calibration method in~\citet{gittins1979bandit}, are given to calculate the Gittins index for the Bernoulli and Gaussian bandit with known variance. 
	
	Calculation methods found in literature only work for countable state FABPs or are tailored to the Gaussian bandit with known variance. The calculation methods often revolve around the Bellman equation where it is assumed that the transition probabilities are known in closed form. First, when dealing with experimental data, the models can be much more complex and the assumption of known transition probabilities might not hold. For instance, the main types of outcome data encountered in clinical trials are categorical, continuous or event-times. Clearly only the first of these three is covered by considering a finite or countable state space, as continuous outcome data are often modeled using parameters on an uncountable parameter space. Second, in, e.g., latent variable modeling, the posterior distribution of the model parameters is often not known in closed form. Markov chain Monte Carlo approaches (\citet{gilks1995markov}) provide a means to do (approximate) posterior inference in this case. Hence, the transition structure of the FABP is unknown, however one can still sample (approximately) from the Markov reward process.  Another situation where the posterior is not available in closed form is when the prior being used in the Bayesian analysis is nonconjugate with respect to the likelihood. For instance, when a Bayesian experiment is performed for Gaussian data with possibly conflicting prior information, a nonconjugate Student's t-distribution can be assumed as a  prior distribution for the mean (\citet{neuenschwander2020predictively}).
	No method exists to accurately approximate the Gittins index in these settings. 
	
	To address the open problems listed above, a sampling-based Gittins index approximation (SBGIA) is introduced in  this paper.  The SBGIA can be calculated for any type of state space, also when there is access only to a simulator for future rewards. We first approximate the Gittins index using  a truncation of the optimization horizon and support for the immediate rewards of the Markov chains. Second, we use
	a stochastic approximation procedure \hbox{(\citet{robbins1951stochastic})} based on 
	the optimal stopping value approximation introduced in~\citet{chen2018beating} to find the root (i.e., the fair charge) of the prevailing charge formulation of the Gittins index, see, e.g., \citet[Equation 2]{weber1992gittins}. As the SBGIA is sampling-based, samples from the (approximate) posterior reward distribution can be used to make decisions in the proposed algorithm. 
	
	The paper is organized as follows. Section \ref{FABP} introduces the family of alternative bandit processes. Section \ref{prelim} extends optimal stopping value approximation results from \hbox{\citet{chen2018beating}} to  reward processes that are not restricted to be non-negative.
	In Section \ref{sec:3.3}, based on convergence results provided in~\citet{chen2018beating}, we obtain finite-time convergence results for the SBGIA.
	In Section \ref{sec:3.4} we prove asymptotic convergence results for the stochastic approximation iterates. In Section \ref{sect:numerical_results} we show the performance of the SBGIA in several numerical simulation studies. Appendix~\ref{proofs} states the longer proofs of the theorems in the paper, and Appendix~\ref{table:notation} summarises the notation used in the paper.
	
	\section{Family of alternative bandit processes~\label{FABP}}
	We consider $A$ independent 
	Markov~chains~$(\bS_t^a)_t$ for $a\in[A]=\{1,\dots, A\}$, referred to as {\it arms}, each on a (shared) Borel space $(\mathcal{S},\mathcal{G})$  with underlying probability space $(\Omega, \mathcal{F},\mathbb{P})$. The (shared) transition kernel is denoted $\mathbb{P}_\bs$, where $\bs$ denotes the initial state of the Markov chain, and expectations w.r.t.~$\mathbb{P}_\bs$ are denoted $\mathbb{E}_\bs$. 
	Let $R$ denote a  $\mathcal{G}/\mathcal{B}(\mathbb{R})$-measurable reward function. \hbox{See~\citet[Chapter 35]{lattimore2020bandit}} for details on this setting. 
	
	Let $\mathcal{P}$ be the set of policies, i.e., the mappings from the set of histories \begin{equation}
		\mathcal{H} = \{(\bs^1_u, \dots, \bs_u^A,a_u)_{u=0}^{t-1}\times (\bs_t^1,\dots, \bs_t^A):t\in\mathbb{N}_0,\; a_u\in[A],\;\bs_u^a\in \mathcal{S}\;\;\;\; \forall a\in [A], u\in\mathbb{N}_0\}\label{defn:histories}
	\end{equation} to the unit-$A$-simplex $\Delta^{{A}}$ of probability vectors over $[A]$. 
	A fixed policy $\pi\in\mathcal{P}$ induces a Markov chain $\left(\left((\bS^a_u)_{u=1}^{N_{a,t}^\pi}\right)_{a\in[A]}\right)_t$, where $N_{a,t}^\pi$ is the number of times arm $a$ is chosen by the policy $\pi$  up to and including time $t$. We denote the  probability measure and expectation under this fixed policy by $\mathbb{P}_\pi$ and  $\mathbb{E}_{\pi}$.
	The objective is to find the optimal policy maximizing the total sum of discounted rewards for a 
	discount factor $\gamma\in(0,1)$: 
	\begin{equation}\pi^* = \underset{\pi\in\mathcal{P}}{\arg\max}\;\mathbb{E}_\pi 
		\left[ \sum_{t=0}^{\infty}\gamma^tR\left(\bS_{N_{A_t,t}}^{A_t}\right)\right]\label{totdiscounted}.\end{equation}
	We furthermore assume that the reward function for arm $a$, with initial state $\bs^a$, is discounted absolutely convergent in expectation under discount factor $\gamma$
	\begin{assumption} \label{boundedtotrew}  \begin{equation*}C(\bs^a)=\mathbb{E}_{\bs^a}\left[\sum_{t=0}^\infty \gamma^t |R(\bS^a_{t})|\right]<\infty \;\;\;\; \forall \bs^a\in\mathcal{S}.\end{equation*}
	\end{assumption}
	Under Assumption \ref{boundedtotrew}, the $\arg\max$ in \eqref{totdiscounted} is attained, and the optimal policy $\pi^*$ for the Markov decision process above is the policy choosing the arm with the highest  Gittins index, \mbox{see, e.g., \citet[Theorem 35.9]{lattimore2020bandit}.} To specify this, let $\bs^a_h$ be the current state for arm $a$ in history $h\in\mathcal{H}$, then
	\begin{equation}\pi^*(h) \in \underset{a\in[A]}{\arg\max} \;\nu(\bs_{h}^a),\quad \mbox{with}\quad \nu(\bs^a) = \underset{\tau\in\mathcal{T}^a}{\sup}\;\frac{\mathbb{E}_{\bs^a}\left[\sum_{t=0}^{\tau-1}\gamma^t R(\bS^a_{t})\right]}{\mathbb{E}_{\bs^a}\left[\sum_{t=0}^{\tau-1}\gamma^t\right]}\;\;\;\; \forall \bs^a\in\mathcal{S},\label{defGI}\end{equation}
	where  $\mathcal{T}^a$ is the set of stopping times in $\mathbb{N}$ w.r.t.~the filtration $(\mathcal{F}^a_{t})_{t}$  generated by the process $(\bS_t^a)_t$ starting from state $\bs^a$. The {\it Gittins index} $\nu(\bs^a_h)$ hence only depends on the current state $\bs^a_h$ for any arm $a$.
	When ties occur in the above expression, i.e., the $\arg\max$ returns more than one value, the policy uniformly chooses an arm  from the different choices of $\arg\max$. Hence, under this choice, the Gittins index policy is a randomized Markov policy.
	
	The Gittins index can also be written as (\citet[Equation 2]{weber1992gittins})
	\begin{equation}\nu({\bs^a}) = \sup\left\{\nu\;:\: \underset{\tau\in\mathcal{T}^a}{\sup}\;\mathbb{E}_{\bs^a}\left[\sum_{t=0}^{\tau-1}\gamma^t(R(\bS^{a}_{t}) - \nu)\right]\geq 0\right\}\label{defGI2}\;\;\;\; \forall {\bs^a}\in\mathcal{S}.\end{equation}
	Under Assumption \ref{boundedtotrew}, the supremum in the above condition is zero at a unique point $\nu({\bs^a})$ \hbox{(\citet{lattimore2020bandit})}, hence $\nu({\bs^a})$ is the root of the function \begin{equation}\nu\mapsto \underset{\tau\in\mathcal{T}^a}{\sup}\;\mathbb{E}_{\bs^a}\left[\sum_{t=0}^{\tau-1}\gamma^t(R(\bS^a_{t}) - \nu)\right]\label{def_func_nu_gittins}.\end{equation}
	
	
	\section{Preliminaries}\label{prelim}
	This section 
	extends the results on the optimal stopping value approximation introduced in~\citet{chen2018beating} to  reward processes that are not restricted to be non-negative, which are needed to develop our results. Section~\ref{sect: theoretical_results} translates these results to the setting of the  family of alternative bandit processes.
	\subsection{Optimal stopping approximation}\label{sec:3.1}

	When considering the behavior of only one of the $A$ arms, the superscript $a$ is dropped from the state, filtration, and set of stopping times.
	Let $\ixv{t} = \{0,1,\dots t\}$, $\bS_{\ixt}$ contain the realisations of $\bS$ up to time $t$,  and let $\mathcal{F}_t$ be the smallest sigma algebra for which $\bS_{\ixt}$ is measurable.
	For $N\in\mathbb{N}$, let $\mathcal{T}_N$ denote the set of integer-valued stopping times $\tau$ adapted to $(\mathcal{F}_t)_{t\in [N]}$ such that $\tau\in[N]$ almost surely.
	We assume that $\mathcal{S}$ is a Polish space, ensuring the existence of regular conditional probabilities for $(\bS_t)_{t}$  \hbox{(e.g., \citet[Theorem 12.3.1]{athreya2006measure}).} 
	Let $Z_t = g_t(\bS_{\ixt})$ for measurable real-valued functions $(g_t)_{t\in[N]}$.
	The random variable $Z_t$  is assumed integrable (on the probability space for $\bS$) for all $t$. The goal is to compute
	\begin{equation}
		\inf_{\tau\in\mathcal{T}_N}\mathbb{E}[Z_\tau].\label{optstop}
	\end{equation}
	The following two results extend Theorems 1 and 2  in~\citet{chen2018beating} to remove the restriction that $Z$ is non-negative. The proofs readily follow along the lines of the proofs in  \citet{chen2018beating}. 
	Theorem 1  \mbox{expresses \eqref{optstop}} as an infinite sum. For $Z_t$ bounded, Theorem 2 provides an error bound for truncation of the infinite sum.
	\begin{theorem}[{\sc Optimal stopping value representation}]\label{thm:infsumapprox}
		\begin{equation*}\inf_{\tau\in\mathcal{T}_N}\mathbb{E}[Z_\tau] = \sum_{k=1}^\infty \mathbb{E}\left[\min_{u\in[N]} Z_u^{(k)}\right]\label{infsumversion},\end{equation*} where for all $k\in\mathbb{N}$ and $t\in [N]$
		\begin{align}
			Z_t^{(1)}&=Z_t,\label{defZ1}\\
			Z_t^{(k+1)}&=Z_t^{(k)} - \mathbb{E}\left[\min_{u\in[N]}Z_u^{(k)}\Big|\mathcal{F}_t\right]\label{defZk}.
		\end{align}
	\end{theorem}
	\begin{proof}
		The proof follows the proof of Lemma 1 and Theorem 1 in~\citet{chen2018beating} noting that, even though $Z$ is not assumed non-negative in this case, the sequence $(Z_t^{(k)})_{k\geq 2}$ remains a non-negative decreasing sequence of random variables for all $t\in [N]$. Hence $\lim_{k\rightarrow\infty}\inf_{\tau\in\mathcal{T}_N}\mathbb{E}[Z_\tau^k]=0$ still holds.  \end{proof}

	\begin{theorem}[{\sc Optimal stopping approximation}]\label{theorem:approxthm}
		Suppose $Z_t\in [a,\,b]$ almost surely for all $t\in[N]$ for some $a,\,b\in\mathbb{R}$ such that $a<b$, then
		\begin{equation}
			0\leq \inf_{\tau\in\mathcal{T}_N}\mathbb{E}[Z_\tau] - \sum_{k=1}^K \mathbb{E}\left[\min_{u\in[N]} Z_u^{(k)}\right]\leq \frac{b-a}{K+1}.\label{Zbound}
		\end{equation} 
	\end{theorem}
	\begin{proof}
		The proof for $a=0,\;b=1$ follows from  \citet[Theorem 2]{chen2018beating}. The extension to general closed intervals follows by considering the following mapping between stochastic processes
		$$T(Z) = \left( Z_t - \mathbb{E}\left[\min_{u\in[N ]}Z_u\Big|\mathcal{F}_t\right]\right)_{t\in\mathbb{N}}. $$
		We have that $T((c_1-c_2) Z + c_3 ) = (c_1-c_2)T(Z)$ for all $c_1,\,c_2,\;c_3\in\mathbb{R}$. Hence, letting \\$\tilde{Z} = (Z-a)/(b-a)$ we have from \citet[Theorem 2]{chen2018beating} that
		$$ \inf_{\tau\in\mathcal{T}_N}\mathbb{E}[Z_\tau^{(k+1)}] =\inf_{\tau\in\mathcal{T}_N}\mathbb{E}[T^{k}({Z})_{\tau}]= (b-a)\inf_{\tau\in\mathcal{T}_N}\mathbb{E}[T^{k}(\tilde{Z})_{\tau}]=(b-a)\inf_{\tau\in\mathcal{T}_N}\mathbb{E}[\tilde{Z}_\tau^{(k+1)}]\leq \frac{b-a}{K+1}.$$
	\end{proof}
	\subsection{Simulation approximation }\label{sec: 3.2}
	Following \citet{chen2018beating}, the sum of expectations in \eqref{Zbound} may be approximated via simulation. Let $[\bv,\bw]$ denote concatenation of the vectors $\bv$ and $\bw$. 
	Let $\bm 1_m$ be the all-ones vector in $\mathbb{R}^m$ for all $m\in\mathbb{N}$.
	For each index $\bi\in\cup_{k=1}^{K}\mathbb{N}^k$ let
	$\bS^{\bi}$ be versions of $\bS$ such that $\bS^{\bi}=\bS^{\bj}$ if $\bj=[\bi,\bm 1_m]$, and $\bS^\bi$ is independent of $\bS^{\bj}$ otherwise.
	
	Let $K\in\mathbb{N}$ and $n(\bi,k) \in \mathbb{N}$ for all $\bi,k$. Recall \eqref{defZ1} and \eqref{defZk}, for all $k<K$ define the random processes 
	\begin{align}
		Z^{(1)}_{\bi, t, n}&= g_t(\bS^{\bi}_{\ixt}), \label{Zsamp1}\\
		Z^{(k+1)}_{\bi, t, n} &= Z^{(k)}_{[\bi,1], t, n} - \frac{1}{n(\bi,k)}\sum_{j=2}^{n(\bi,k)+1}\left(\min_{u\in[N]}Z^{(k)}_{[\bi,j],u,n}\Big|\{\bS^{[\bi,j]}_{\ixt} = \bS^{\bi}_{\ixt}\}\right), \label{Zsamp2}
	\end{align}
	where $Z^{(k)}_{[\bi,j],u,n}\Big|\{\bS^{[\bi,j]}_{\ixt} = \bS^{\bi}_{\ixt}\}$ denotes the random variable $Z^{(k)}_{[\bi,j],u,n}$ conditioned on the event that the paths of the two processes $\bS^{[\bi,j]}$ and $\bS^{[\bi]}$ are equal up to time  $t$.  After time $t$ they continue independently (as $j\geq 2$). Note that this random variable is well defined as the regular conditional probabilities for $(\bS_t)_t$ exist (Section \ref{sec:3.1}).  The requirement $\bS^{\bi}=\bS^{\bj}$ if $\bj=[\bi,\bm 1_m]$ induces that the  ``right" partial paths of $\bS^\bi$ used at a lower level are used to construct values $Z_{\bi,t,n}^{(k)}$ at higher \mbox{level $k$}. To illustrate our notation and the relation between  \eqref{defZ1}, \eqref{defZk} and  \eqref{Zsamp1}, \eqref{Zsamp2}, note that $Z^{(2)}_{[2,j_1],t,n}$ equals
	\begin{align*}
		&Z^{(1)}_{[2,j_1,1],t,n} - \frac{1}{n([2,j_1],1)}\sum_{j_2=2}^{n([2,j_1],1)+1}\left(\min_{u\in[N]}Z^{(1)}_{[2,j_1,j_2],u,n}\Big|\{\bS^{[2,j_1,j_2]}_{\ixt} = \bS^{[2,j_1]}_{\ixt}\}\right)\\
		& \approx\; Z^{(1)}_{[2,j_1,1],t,n} - \mathbb{E}\left[\min_{u\in[N]}Z_{[2,j_1,1],u, n}^{(1)}\Big|\{\bS^{[2,j_1,1]}_{\ixt} = \bS^{[2,j_1]}_{\ixt}\}\right] = Z^{(1)}_{[2,j_1,1],t,n} - \mathbb{E}\left[\min_{u\in[N]}Z_{[2,j_1,1],u, n}^{(1)}\Big|\mathcal{F}_t^{[2,j_1,1]}\right],
	\end{align*}
	where
	$Z^{(1)}_{[2,j_1,j_2],t,n} = g_t(\bS^{[2,j_1,j_2]}_{\ixt})$, $\mathcal{F}_t^{[2,j_1,1]}$ is the sigma algebra generated by $\bS_{\ixt}^{[2,j_1,1]}$, and the approximation is exact in the limit $n([2,j_1],1)\rightarrow \infty$ by the law of large numbers. 
	Note that the process $\bS^{[2,j_1,1]}$ used to determine $Z^{(1)}_{[2,j_1,1],t,n}$ equals $\bS^{[2,j_1]}$. Hence, corresponding to conditioning on $\mathcal{F}_t$ above, in order to determine $Z^{(2)}_{[2,j_1],t,n}$ it is assumed that the history of the Markov process up to time $t$ is the same for all $Z^{(1)}_{[2,j_1,j_2],t,n}$ after which the processes continue independently. 
	
	Consider the following random variable projected on $[a,b]$:
	\begin{equation}V_{n}^{(K)} = \max\left(a,\;\min\left(b,\;
		\frac{1}{n(K,K)}\sum_{j=1}^{n(K,K)}\sum_{k=1}^K \min_{u\in [N]}Z_{[k,j],u, n}^{(k)}\right) \right).\label{eqn:Vsamp}\end{equation}
	By the law of large numbers and Theorem~\ref{thm:infsumapprox}, we have that  $$\lim_{K\rightarrow\infty}\lim_{n\rightarrow  \infty}V_n^{(K)}= \lim_{K\rightarrow\infty}\sum_{k=1}^K \mathbb{E}\left[\min_{u\in[N]} Z_u^{(k)}\right]=\inf_{\tau\in\mathcal{T}_N}\mathbb{E}[Z_\tau],$$
	where $n\rightarrow\infty$ indicates the element-wise limit $n(\bi,k)\rightarrow \infty$ for all $(\bi,k)$.

	Following \citet{chen2018beating}, we may  show that for every $\xi>0$ and $\delta\in(0,1)$ we can choose a function $n_{\xi,\delta}$, given implicitly in Algorithms $\mathcal{B}^k$ and $\hat{\mathcal{B}}^k$ in~\hbox{\citet[Pages 27 and 30]{chen2018beating}},  such that 
	$$\mathbb{P}\left(\left|\sum_{k=1}^K\mathbb{E}\left[\min_{u\in[N]}Z_u^{(k)}\right]-  V_{n_{\xi,\delta}}^{(K)}  \right|\leq \xi/2\right)\geq 1-\delta.  $$
	This statement was  shown in~\citet{chen2018beating} for non-projected $V_{n_{\xi,\delta}}^{(K)}$. However, the projection on $[a,b]$ can only reduce the error $|\sum_{k=1}^K\mathbb{E}[\min_{u\in[N]}Z_u^{(k)}]-  V_{n_{\xi,\delta}}^{(K)} |$  so the statement still holds for $V_{n_{\xi,\delta}}^{(K)}$ as defined in \eqref{eqn:Vsamp}.
	Hence, choosing $K(\xi) = \lfloor2(b-a)/\xi\rfloor$ we have by Theorem~\ref{theorem:approxthm} and the triangle inequality that
	$$\mathbb{P}\left(\left|\inf_{\tau\in\mathcal{T}_N}\mathbb{E}[Z_\tau]-  V_{n_{\xi,\delta}}^{(K(\xi))}  \right|\leq \xi\right)\geq 1-\delta.  $$
	\section{Gittins index approximation}~\label{sect: theoretical_results}
	This section first introduces our main method, which is a sampling-based method for Gittins index approximation.  Section \ref{sec:3.3} develops finite-time bounds for the approximation, and Section \ref{sec:3.4} develops asymptotic convergence results.

	Combining the results of Sections  \ref{FABP} and \ref{sec:3.1}, we define for some  $R_\ell, R_u\in\mathbb{R}$ such that $R_\ell<R_u$ 
	\begin{equation}
		\tilde{Z}_t(\nu) = g^\nu_t(\bS_{\ixt}) =\frac{1-\gamma}{2(R_u-R_\ell)(1-\gamma^N)} \sum_{u=0}^{t-1}\gamma^u (\nu-R(\bS_u)),\label{Z_unn}
	\end{equation} 
	which is the argument in \eqref{def_func_nu_gittins}
	, scaled by $c=\frac{1-\gamma}{2(R_u-R_\ell)(1-\gamma^N)}$ for later convenience.
	
	We now introduce our sampling-based Gittins index approximation.\\
	
	{\sc Sampling-based Gittins index approximation (SBGIA)}
	\begin{itemize}
		\item {\bf Approximation I: truncation}\\
		Truncate the support of $\tilde{Z}_t$ and the time horizon for the optimal stopping problem~\eqref{defGI2}:
		\begin{equation}\nu_\sigma(\bs)= \sup\left \{\nu:\inf_{\tau\in\mathcal{T}_\sigma}\mathbb{E}_\bs[\tilde{Z}_\tau(\nu)]\leq 0 \right\}\approx \sup\left \{\nu:\sup_{\tau\in\mathcal{T}}\mathbb{E}_\bs\left[\sum_{t=0}^{\tau-1}\gamma^t (R(\bS_t)-\nu)\right]\geq 0 \right\} ,\label{truncGI}\end{equation}
		where the infimum is taken over the set of stopping times $\mathcal{T}_\sigma=\{\tau\in\mathcal{T}:\tau\leq \sigma\}$
		with, for a choice of $N\in\mathbb{N}$,
		\begin{equation}
			\sigma = \sigma_H\wedge N, \;\;\;\; \sigma_H = \inf\{t\in\mathbb{N}: R(\bS_t)\notin [R_\ell,R_u]\}\label{def:stoppingtime}
		\end{equation}
		where the minimum operator is denoted with $\wedge$.
		Note that $\nu_\sigma$ only considers $R(\bS_t)\in [R_\ell,R_u]$. Hence,  using the stopped (bounded) process $Z=\tilde{Z}^{\sigma} $ such that $\tilde{Z}_t^{\sigma}=Z_{t\wedge\sigma}$ for all $t$, $\nu_\sigma$ in \eqref{truncGI}  can also be formulated  as
		\begin{equation}\nu_{\sigma}(\bs)=\sup\left \{\nu:\inf_{\tau\in\mathcal{T}_N}\mathbb{E}_\bs[Z_\tau(\nu)]\leq 0 \right\}.\label{final_approx_GI}\end{equation}
		Boundedness of $Z$ allows using the results stated in Section~\ref{sec:3.1}.
		\item {\bf Approximation II: simulation}\\
		Highlighting the dependence on the state $\bs$ and current estimate $\nu$ only, we sample $V_{\bs}(\nu)~=~V^{(K(\xi),N)}_{\bs,n_{\xi,\delta}}(\nu)$  truncated to $[-1/2,\,1/2]$ by sampling the respective processes $\bS^\bi$, from the Markov kernel $\mathbb{P}_{\bs}$ starting from state $\bs$,  needed to determine $Z^{(K(\xi))}_{\bi,t,n_{\xi,\delta}}$ in  \eqref{Zsamp1}, \eqref{Zsamp2} and combining them in \eqref{eqn:Vsamp} such that
		\begin{equation}\mathbb{P}\left(\Bigg| \inf_{\tau\in\mathcal{T}_N}\mathbb{E}_\bs[Z_\tau(\nu)] -V_{\bs}(\nu) \Bigg|\leq \xi\right) \geq 1-\delta.\label{PACbound_V}\end{equation}
		Using this sampling procedure, we approximate $\nu_\sigma(\bs)$ using stochastic approximation  \hbox{(\citet{borkar2009stochastic}),} i.e., a stochastic root-finding procedure, starting from an initial point \hbox{$\nu_0(\bs)\in[R_\ell,R_u]$} such that
		\begin{equation}
			\nu_{m+1}(\bs) = \nu_m(\bs) - \alpha_mV_{\bs,m}(\nu_m(\bs)),\label{stoch_approx}
		\end{equation}
		where $V_{\bs,m}$ are independent versions of $V_{\bs}$, and $(\alpha_m)_m$ is a possibly stochastic, predictable non-negative sequence of step-sizes in $\mathbb{R}$. We collect $\nu_M(\bs)$ as our sampling-based approximation of the Gittins index $\nu(\bs)$, where $M$ is defined according to a certain (user-defined) stopping criterion (see Remark~\ref{rem:criterion}).
	\end{itemize}

	\subsection{Finite-time error bounds}\label{sec:3.3}
	In this section, we first derive a truncation error bound for the first-stage approximation (Theorem~\ref{Thm:Error_cutoff}). Subsequently, we couple the stochastic approximation iterates from the second-stage approximation  to stochastic approximation iterates  from a continuous increasing function (Lemma~\ref{lemma:coupled_boven_onder}). Using a mean-squared error recursion result for these coupled sequences (Theorem  \ref{theorem:FTCB}), we then construct a confidence interval for the Gittins index in finite-time, where ``finite-time'' pertains to the number of stochastic approximation iterates~(Theorem~\ref{FTCI}). Using finite-time bounds, we construct an $\epsilon$-optimal policy for the family of alternative bandit processes~(Theorem~\ref{Thm:epsilonoptimal}).
	
	An upper bound on the error when approximating $\nu(\bs)$ by the truncation-based index $\nu_\sigma(\bs)$ as defined in \eqref{truncGI} is given in the following theorem, which holds for a  general stopping time $\sigma$. A similar bound was given in~\citet{wang1997error} for a fixed truncation $N$ of the time horizon. 
	\begin{theorem}[{\sc Truncation error bound}] \label{Thm:Error_cutoff}
		\thmBoundtrunc{defGI_restricted}
	\end{theorem}
	\begin{proof}
		The complete proof can be found in Appendix \ref{proofs} and is outlined here. The first two statements follow
		by showing that 
		$\nu_\sigma(\bs)$ equals the Gittins index for the Markov process $\tilde{\bS}$ documenting the full history of $\bS$ up to and including each time $t$ with rewards $\tilde{R}(\tilde{\bS}_t) = R(\bS_t)\mathbb{I}(t< \sigma)$, after which we can apply commonly known results for the Gittins index from \hbox{\citet{lattimore2020bandit}.}
		Only the upper bound $\nu(\bs)-\nu_\sigma(\bs)\leq \mathbb{E}_\bs\left[\gamma^\sigma \nu(\bS_\sigma)^+\right]/(1 -\mathbb{E}_\bs[\gamma^\sigma])$ is non-trivial and mainly follows from bounding the difference between two optimal stopping values by the optimal stopping value of the difference and using the strong Markov property.  
	\end{proof}
	We now give an error bound  for the sampling-based approximation $\nu_M(\bs)$ as defined in \eqref{stoch_approx}.
	Note that $ Z_{t}(\nu)\in[-1/2,1/2]
	$ almost surely. Hence, according to Theorem~\ref{theorem:approxthm} we have 
	\begin{equation*}
		0\leq \inf_{\tau\in\mathcal{T}_N}\mathbb{E}_\bs\left[Z_{\tau}(\nu)\right] - \sum_{k=1}^K\mathbb{E}_\bs\left[\min_{u\in[N]}Z
		^{(k)}_{u}(\nu)\right]\leq \frac{1}{K+1}.\label{ineq:inf_approx}
	\end{equation*} 
	Defining the functions \begin{equation}\tilde{f}_\bs:\nu\mapsto \mathbb{E}_\bs[V_\bs(\nu)],\;\;\;\; f_\bs:\nu \mapsto \inf_{\tau\in\mathcal{T}_N}\mathbb{E}_\bs[Z_\tau(\nu)],\label{def_f}\end{equation}
	we have by \eqref{PACbound_V} and by Jensen's inequality 
	\begin{equation}|\tilde{f}_\bs(\nu)-f_\bs(\nu)|\leq \mathbb{E}_\bs\Bigg| \inf_{\tau\in\mathcal{T}_N}\mathbb{E}_\bs[Z_\tau(\nu)] -V_{\bs}(\nu) \Bigg|\leq \delta + \xi =:B(\delta,\xi).\label{bound_mean_diff}\end{equation}

	Using these results, we derive mean-squared error bounds for our approximation method.
	We do this by defining coupled stochastic approximation iterates based on the function $f_\bs$ that lie almost surely below and above the sequence  $(\nu_m(\bs))_m$ defined in \eqref{stoch_approx}.
	As we do not have access to $f_\bs$ or even unbiased estimates of $f_\bs$, these iterates cannot directly be simulated, but  we can bound their distance to $(\nu_m(\bs))_m$, and show finite-time convergence results for these sequences, which can then be used to derive finite-time convergence bounds for $(\nu_m(\bs))_m$.
	To this end, {let} \begin{equation} \epsilon_m= V_{\bs,m}(\nu_m(\bs))-\tilde{f}_\bs(\nu_m(\bs)),\label{MartDiffSeq}\end{equation}
	which is a martingale difference sequence with respect to the natural filtration $(\mathcal{F}^{\epsilon}_m)_m$ w.r.t. $(\epsilon_m)_m$ as $V_{\bs,m}$ (hence $\tilde{f}_\bs$) is bounded and $$\mathbb{E}[\epsilon_m\mid\mathcal{F}_{m-1}^\epsilon] = \mathbb{E}[V_{\bs,m}(\nu_m(\bs))\mid\nu_m(\bs)]-\tilde{f}_\bs(\nu_m(\bs))=0,$$ since $\nu_m(\bs)$ is a function of $(\epsilon_{m^\prime})_{m^\prime = 1}^{m-1}$.

	Let  
	\begin{align}
		\bar{\nu}_0(\bs)={\nu}_0(\bs),\;\;\;\; \bar{\nu}_{m+1}(\bs) &= \bar{\nu}_{m}(\bs)  - \alpha_m\left(f_{\bs}(\bar{\nu}_m(\bs))- B(\delta,\xi) + \epsilon_m \right),\label{defupperSA}\\
		\ubar{\nu}_0(\bs)={\nu}_0(\bs),\;\;\;\; \ubar{\nu}_{m+1}(\bs) &= \ubar{\nu}_{m}(\bs)  - \alpha_m\left(f_{\bs}(\ubar{\nu}_m(\bs)) + B(\delta,\xi) + \epsilon_m \right),\label{deflowerSA}
	\end{align}
	where $(\alpha_m)_m$ is the same sequence as in \eqref{stoch_approx}.
	\newpage
	\begin{lemma}\label{lemma:coupled_boven_onder} Assume the step-size sequence $(\alpha_m)_m$ is such that $\sup_m \alpha_m\leq 2(R_u-R_\ell)$ almost surely. For $m=0,1,2,\ldots$, we have
		\begin{equation}
			\ubar{\nu}_m\leq \nu_m\leq \bar{\nu}_m.\label{eq:bounds}
		\end{equation}
	\end{lemma} 
	\begin{proof}Observe that $\bar{\nu}_0=\ubar{\nu}_0 = \nu_0$, so that \eqref{eq:bounds} is satisfied for $m=0$. 
		\\Now assume $\ubar{\nu}_m\leq \nu_m\leq \bar{\nu}_m$ for some $m\geq 0$. We have
		\begin{align*}\bar{\nu}_{m+1}(\bs)-{\nu}_{m+1}(\bs) &= \bar{\nu}_{m}(\bs)-{\nu}_{m}(\bs) - \alpha_m\left(f_{\bs}(\bar{\nu}_{m}(\bs)) -(\tilde{f}_{\bs}({\nu}_{m}(\bs))+B(\delta,\xi)) \right)\\&\geq\bar{\nu}_{m}(\bs)-{\nu}_{m}(\bs) - \alpha_m\left(f_{\bs}(\bar{\nu}_{m}(\bs)) -{f}_{\bs}({\nu}_{m}(\bs)) \right)\\&\geq \left(1- \frac{\alpha_m}{2(R_u-R_\ell)}\right)\left(\bar{\nu}_{m}(\bs)-{\nu}_{m}(\bs)\right)\geq 0,
		\end{align*}
		where the second statement follows from \eqref{bound_mean_diff}, and the last statement follows as for any $\nu_1,\nu_2\in\mathbb{R}$
		\begin{equation}
			\frac{\min\left(\frac{1-\gamma}{1-\gamma^N}(\nu_1-\nu_2),\nu_1-\nu_2\right)}{2(R_u-R_\ell)}\leq  f_\bs(\nu_1)-f_\bs(\nu_2) \leq  \frac{\max\left(\frac{1-\gamma}{1-\gamma^N}(\nu_1-\nu_2),\nu_1-\nu_2\right)}{2(R_u-R_\ell)},\label{Ineq: bounds_f}
		\end{equation}
		which can easily be shown from the definition of $f_\bs$, using the fact that the suprema are attained at unique stopping times.

		Similarly, if $\ubar{\nu}_m(\bs)\leq {\nu}_m(\bs)$ for some $m$, then $\ubar{\nu}_{m+1}(\bs)\leq {\nu}_{m+1}(\bs)$. The proof is completed by induction.
	\end{proof}
	
	~
	
	The next theorem gives the mean-squared error between either sequence $\bar{\nu}_m(\bs),\; \ubar{\nu}_m(\bs)$ and a limit point. 
	{ \begin{theorem}[\sc Mean-squared error recursion for coupled sequences]\label{theorem:FTCB}
			\thmFTCB
	\end{theorem}}
	\begin{proof}
		The full proof can be found Appendix \ref{proofs}. The result is obtained by using conditional independence of the martingale difference sequence $\epsilon_m$, rewriting the difference between $f_\bs(\omega^*)$ and $f_\bs(\omega_m)$ to a scaled difference between $\omega^*$ and $\omega_m$ using \eqref{Ineq: bounds_f} and the law of total expectation.
	\end{proof}
	
	~
	
	Observe that by the assumption on $\alpha_m$ we have $c\alpha_m<1$ (for $N>1$) so that the influence of the initial difference between $\omega_0$ and $\omega^*$ decays exponentially. The squared difference of the iterate $\omega_m$ and $\omega^*$ converges to a value that depends on the second moment of the martingale differences and the step-size sequence $(\alpha_m)_m$.
	
	Different choices of step-size sequences yield different upper bounds on the rates of convergence.  We examine two standard choices.
	\begin{example}[Constant step-size]
		Let $\alpha_m=\alpha$ for all $m$ with $\alpha\leq 2(R_u-R_\ell)$.
		Let $B_0^2 = (\omega_0-\omega^*)^2$, and assume $\max_m\mathbb{E}[\epsilon_m^2]\leq v_\epsilon^2.$
		From the recursion for $\omega_m$, we obtain that
		$$\mathbb{E}[(\omega^*-\omega_{m})^2]\leq B^2_0(1-c\alpha)^{2m} + \frac{v^2_\epsilon \alpha}{c}.$$
		
	\end{example}
	\begin{example}[Linear step-size]
		Let $\alpha_m=A/m$ for all $m$, with $A> 1/c$. Assume $\max_m\mathbb{E}[\epsilon_m^2]\leq v_\epsilon^2.$ Then, by Theorem~\ref{theorem:FTCB} we have for all $m_0\geq A/(2(R_u-R_\ell))$ that
		$$\mathbb{E}[(\omega^*-\omega_{m+1})^2]\leq (1-cA/m)^2\mathbb{E}[(\omega^*-\omega_{m})^2] + A^2v_\epsilon^2/m^2\leq (1-cA/m)\mathbb{E}[(\omega^*-\omega_{m})^2] + A^2v_\epsilon^2/m^2.$$
		Hence, by \citet[Lemma 1]{chung1954stochastic} we have for all $m\geq A/(2(R_u-R_\ell))$ that
		$$\mathbb{E}[(\omega^*-\omega_{m})^2]\leq \frac{A^2v_\epsilon^2}{m(cA-1)} + o(m^{-2} + m^{-cA}).$$
	\end{example}
	In the above two examples, we see that we can choose the step-size sequence and stopping point $m$ such that $\mathbb{E}[(\omega^*-\omega_{m})^2]$ is arbitrarily small. We can construct confidence intervals for $\omega^*$ using Chebyshev's inequality:
	$$\mathbb{P}(|\omega^*-\omega_m|>\xi) \leq \frac{\mathbb{E}[(\omega^*-\omega_{m})^2]}{\xi^2}\leq \delta.$$
	
	Observe that $f_\bs$ is a strictly increasing surjective function on $\mathbb{R}$. {We may therefore }
	\begin{equation} \mbox{choose  $\bar{\nu}(\bs)$, $\ubar{\nu}(\bs)$ such that  }f_{\bs}(\bar{\nu}(\bs))=B(\delta,\xi), \quad f_{\bs}(\ubar{\nu}(\bs))=-B(\delta,\xi).
		\label{eq:surjective}
	\end{equation} 
	We may now construct a finite-time confidence 
	interval for the Gittins index $\nu(\bs)$  for finite~$m$, which can be made arbitrary small  by choosing $\mathbb{E}_\bs[\sigma],m$ large and $\xi,\delta$ small.  
	
	\begin{theorem}[\sc Finite-time confidence interval]\label{FTCI}
		Let $\bar{\nu}(\bs)$, $\ubar{\nu}(\bs)$ be chosen according to~\eqref{eq:surjective}.  Then
		\begin{equation}
			\bar{\nu}(\bs)-\ubar{\nu}(\bs)\leq 2B(\delta,\xi)/c,\label{radius_omega}
		\end{equation}
		which can be made arbitrarily small by suitable choice of $\xi$ and $\delta$.
		
		Let $(\alpha_m)_m$ and $M$ be chosen such that $\sup_m\alpha_m\leq 2(R_u-R_\ell)$, and  \begin{equation}\mathbb{P}(|\bar{\nu}(\bs) - \bar{\nu}_M(\bs)|\leq \xi_2)\geq 1-\delta_2/2,\;\;\;\;  \mathbb{P}(|\ubar{\nu}(\bs) - \ubar{\nu}_M(\bs)|\leq \xi_2)\geq 1-\delta_2/2.\label{confint_nubars}\end{equation}
		Then,  with probability at least $1-\delta_2$
		\begin{equation}\nu(\bs)\in\left(\nu_M(\bs)-\xi_2-2B(\delta,\xi)/c,\;\nu_M(\bs)+\xi_2+2B(\delta,\xi)/c + \frac{\mathbb{E}_\bs[\gamma^\sigma \nu(\bS_\sigma)^+]}{1-\mathbb{E}_\bs[\gamma^\sigma]} \right).\label{result:confint}\end{equation}
	\end{theorem}
	\begin{proof}
		From \eqref{Ineq: bounds_f} we obtain that 
		\begin{align*}
			2B(\delta,\xi) = f_\bs(\bar{\nu}(\bs)) - f_\bs(\ubar{\nu}(\bs))&\geq c(\bar{\nu}(\bs)-\ubar{\nu}(\bs)),
		\end{align*}
		which implies \eqref{radius_omega}.
		
		Observe that by \eqref{eq:surjective}, $\nu_\sigma(\bs) \in[\ubar{\nu}(\bs), \;\bar{\nu}(\bs)]$ as $f_\bs$ is increasing and $f_\bs(\nu_\sigma(\bs))=0$.  Hence, by a union bound for the complements of the events in \eqref{confint_nubars}, and inequality \eqref{radius_omega}, we have with probability larger than $1-\delta_2$ that the following confidence interval holds for the second-stage approximation
		\begin{equation*}\nu_\sigma(\bs)\in(\nu_m(\bs)-\xi_2-2B(\delta,\xi)/c,\;\nu_m(\bs)+\xi_2+2B(\delta,\xi)/c).\label{CI_nu_N}\end{equation*}
		The result follows 
		from Theorem~\ref{Thm:Error_cutoff}.   \end{proof}

	The confidence interval \eqref{result:confint} can be used to construct an $\epsilon$-optimal policy for any FABP. For this we first need to show that the values of $(\nu_m)_m$ are restricted to a closed bounded interval. The proof of this lemma can be found in Appendix~\ref{proofs}.
	
	\begin{lemma}\label{lemma_nubound}
		\Lemmanubound
	\end{lemma}
	For the family of alternative bandit processes, we can now define an $\epsilon$-optimal policy, which we will denote the SBGIA policy (SBGIAP).
	\begin{theorem}[\sc $\epsilon$-optimal policy for FABP]~\label{Thm:epsilonoptimal}
		Let $\sup_{m}\alpha_m\leq 2(R_u-R_\ell)<\infty$ almost surely. Let $\pi$ be the randomized Markov policy such that for all histories $h\in\mathcal{H}$
		$$\pi(h) = \underset{a\in[A]}{\arg\max}\;\nu_{M}(\bs^a_h),$$
		where $(\nu_m)_m$ is determined by \eqref{truncGI} -- \eqref{stoch_approx}, and $M,\sigma,\delta,\xi$ are chosen such that 
		for some $\epsilon>0$ 
		\begin{equation}
			\mathbb{P}\left(|\nu(\bs_h^a) - \nu_M(\bs_h^a)|\leq (1-\gamma)^2\epsilon/4\right) ~
			\geq ~ 1-(1-\gamma)^2\epsilon/(4AD(\bs^a_h))\;\;\;\;\forall a\in [A] ,\label{bound_GI_simultaneous}    \end{equation}
		for $D(\bs^a_h)=(1-\gamma)C(\bs_h^a) + \max(|R_\ell|,|R_u|)+M_{\alpha}/2$.Then, the policy $\pi$ is $\epsilon$-optimal for the family of alternative bandit processes, i.e.,
		$$\mathbb{E}_{\pi^*} \left[ \sum_{t=0}^{\infty}\gamma^tR\left(\bS_{N_{A_t,t}}^{A_t}\right)\right] - \mathbb{E}_{\pi} \left[ \sum_{t=0}^{\infty}\gamma^tR\left(\bS_{N_{A_t,t}}^{A_t}\right)\right]\leq \epsilon.$$
	\end{theorem}
	\begin{proof}Note that \eqref{bound_GI_simultaneous} is possible due to \eqref{result:confint}.
		in~\citet{glazebrook1982evaluation}, it is shown that for any stationary policy~$\pi$ 
		\begin{equation}\mathbb{E}_{\pi^*} \left[ \sum_{t=0}^{\infty}\gamma^tR\left(\bS_{N_{A_t,t}}^{A_t}\right)\right] - \mathbb{E}_{\pi} \left[ \sum_{t=0}^{\infty}\gamma^tR\left(\bS_{N_{A_t,t}}^{A_t}\right)\right]\leq \mathbb{E}_{\pi}\left[\sum_{t=0}^\infty\gamma^t\left(\max_{a\in[A]}\nu(\bS_{N_{a,t}}^a) - \nu\left(\bS_{N_{A_t,t}}^{A_t}\right)\right)\right]/(1-\gamma).\label{glazebrookbound}\end{equation}
		By our choice of $m,\sigma,\delta,\xi$ we have by a union bound  over $a$ for the {events} in \eqref{bound_GI_simultaneous} that for any time point $t$
		$$\mathbb{P}\left(\max_{a}|\nu(\bS_t^a) - \nu_m(\bS_t^a)|> (1-\gamma)^2\epsilon/4\right) \leq (1-\gamma)^2\epsilon/(4D(\bS^a_t)).$$
		Combining Lemma~\ref{lemma_nubound} with the fact that $\nu(\bS_t^a)\leq (1-\gamma)C(\bS_t^a)$ yields $|\nu(\bS_t^a) - \nu_m(\bS_t^a)|\leq D(\bS^a_t)$ for all $m$,  hence
		$$\mathbb{E}\left[\max_{a}|\nu(\bS_t^a) - \nu_m(\bS_t^a)|\right]\leq (1-\gamma)^2\epsilon/4 +(1-\gamma)^2\epsilon/4 = (1-\gamma)^2\epsilon/2. $$
		Using this in the right-hand side of \eqref{glazebrookbound} scaled by $(1-\gamma)$ gives
		\begin{align*} 
			\mathbb{E}_{\pi}\left[\sum_{t=0}^\infty\gamma^t\left(\max_{a\in[A]}\nu(\bS_t^a) - \nu\left(\bS_t^{A_t}\right)\right)\right] &=\mathbb{E}_{\pi}\left[\sum_{t=0}^\infty\gamma^t\left(\max_{a\in[A]}\nu(\bS_t^a)-\nu_{m}(\bS^{A_t}_t)+\nu_{m}(\bS^{A_t}_t) - \nu\left(\bS_t^{A_t}\right)\right)\right]\\[1mm]
			&\leq \mathbb{E}_{\pi}\left[\sum_{t=0}^\infty\gamma^t\left(\max_{a\in[A]}\left(\nu(\bS_t^a)-\nu_{m}(\bS^{a}_t)\right)+\nu_{m}(\bS^{A_t}_t) - \nu\left(\bS_t^{A_t}\right)\right)\right]\\
			&\leq \sum_{t=0}^\infty\gamma^t\mathbb{E}_{\pi}\left[2\max_{a\in[A]}|\nu(\bS_t^a)-\nu_{m}(\bS^{a}_t)|\right]\leq(1-\gamma)\epsilon.\end{align*}
	\end{proof}
	\subsection{Asymptotic convergence results\label{sec:3.4}}\phantom{.} This section  investigates convergence properties of the stochastic process defined in \eqref{stoch_approx}. In the previous section, we have shown that
	after a finite, known, amount of iterations, 
	the iterates $\nu_m(\bs)$ lie in an interval containing $\nu(\bs)$ with high probability. The length of this interval depends on the choice of $K$ and $n$. In this section, we let $K$ and $n$ be constants, not depending on $\xi,\delta$, and investigate convergence and asymptotic normality of the iterates $\nu_m(\bs)$ when $m$ goes to infinity, under different choices of the step-size sequence.  First, we show that when the step-size sequence almost surely satisfies the \mbox{Robbins-Monro conditions (\citet{robbins1951stochastic})}, the stochastic approximation iterates converge almost surely (Theorem~\ref{convgRM}). Then, we show that if we instead take a constant step-size sequence, the iterates converge in mean-square to the set of roots (Theorem~\ref{Thm:constantSZ}). Lastly, under stronger conditions, we show that we can construct an adaptive stochastic approximation procedure (\citet{lai1979adaptive}), where a central limit theorem holds for the stochastic approximation iterates (Theorem~\ref{thm:convg_dist}).
	In the following, we let $V_{\bs}(\nu) = V^{(K,N)}_{\bs,n}(\nu)$ for fixed choices of $K,N,n$, and define $\tilde{f}_\bs$ and $\epsilon_m$ as in \eqref{def_f}, \eqref{MartDiffSeq}, respectively.

	The sequence \eqref{stoch_approx} relates to an Euler scheme for the first-order scalar autonomous ordinary~differential~equation~(ODE)~(\citet{borkar2009stochastic}),
	\begin{equation}
		\frac{d}{dx}\nu(x) = -\tilde{f}_{\bs}(\nu(x)) .\label{ODE}
	\end{equation}
	This differential equation has as equilibrium points the roots $\mathcal{R}_{\bs}$ of the function $\tilde{f}_{\bs}$, provided that a root exists.
	{We show that  iteration scheme \eqref{stoch_approx} satisfies the conditions stated in  \hbox{\citet[Chapter~2]{borkar2009stochastic},} hence almost surely the limiting behavior of the sample paths of the stochastic process \eqref{stoch_approx} equals that of the solution to the above ODE.
		We then show that the ODE \eqref{ODE} always converges to an equilibrium point, irrespective of the starting point. It follows that the sample paths of \eqref{stoch_approx} almost surely converge to a random variable $\tilde{\nu}(\bs)\in\mathcal{R}_\bs$ such that $\tilde{f}_{\bs}(\tilde{\nu}(\bs))=0$.}
	
	To make this formal, we define an {\it internally chain transitive invariant set} corresponding to an ODE in accordance with the definition given in~\citet{borkar2009stochastic}. 
	
	\begin{definition}[Internally chain transitive invariant set (\citet{borkar2009stochastic})]
		A closed set $E\subset \mathbb{R}$ is said to be an internally chain transitive invariant set  for the ODE \eqref{ODE}  if
		\begin{itemize} 
			\item any trajectory
			$\nu$ of \eqref{ODE} with $\nu(0) \in E$
			satisfies $\nu(x) \in E$ $\forall x \in \mathbb{R}$, 
			\item  for any $\nu, \nu^\prime  \in E$ and any $\epsilon> 0,\, T > 0,$ there exist $n \geq 1$
			and points $\nu_0 = 
			\nu,\nu_1\dots ,\nu_{n-1}, \nu_n = \nu'$ in $E$ such that the trajectory of \eqref{ODE}
			initiated at $\nu_i$ meets with the $\epsilon$-neighbourhood of $\nu_{i+1}$ for $0 \leq i < n$ after a
			time $\geq T$.
		\end{itemize}
		%
		%
	\end{definition}
	
	~
	
	We first show that $\tilde{f}_\bs$ is Lipschitz continuous.
	\begin{lemma}\label{lemma:properties}
		The function $\tilde{f}_{\bs}$ is Lipschitz continuous in $\nu$. 
	\end{lemma}

	\begin{proof}
		We first prove by induction in $k$ that $Z_{\bi,u,n}^{(k)}$ is Lipschitz continuous  for all $k,\bi,u,n$, starting with $k=1$.
		From 
		\eqref{Zsamp1} and \eqref{Z_unn} 
		it follows that $Z_{\bi, u,n}^{(1)}$ 
		is Lipschitz continuous with constant \mbox{$L=1/(2(R_u-R_\ell))$} for all $\bi,u,n$. 
		Now assume that $Z_{\bi,u,n}^{(k)}$ is Lipschitz continuous up to some $k$ for all $\bi,u,n$.
		As the minimum, average, and sum of a finite set of Lipschitz continuous functions are Lipschitz continuous, we have by \eqref{Zsamp2} that $Z_{\bi, u,n}^{(k+1)}$ is also Lipschitz continuous for all possible $\bi,u,n$.
		By induction we have that $Z_{\bi, u,n}^{(k)}$ is Lipschitz continuous for all $k,\bi,u,n$. Now, by \eqref{eqn:Vsamp}
		$$\tilde{f}_s(\nu) = \mathbb{E}_\bs\left[\max\left(-1/2,\;\min\left(1/2,\;\frac{1}{n(K,K)}\sum_{j=1}^{n(K,K)}\sum_{k=1}^K \min_{u\in [N]}Z_{[k,j],u, n}^{(k)}(\nu)\right)\right)\right]. $$
		Hence, as the expectation, maximum, minimum, and sum of Lipschitz continuous functions are Lipschitz continuous, we have that $\tilde{f}_\bs$ is Lipschitz continuous.  \end{proof}
	Using this lemma, the next result follows.
	\begin{theorem}[\sc Almost sure convergence of stochastic approximation iterates]\label{convgRM}
		\phantom{.}\\ Assume $(\alpha_m)_m$ almost surely satisfies the Robbins-Monro conditions (\citet{borkar2009stochastic}):
		\begin{equation}\sum_{m=1}^\infty \alpha_m = \infty,\;\;\;\; \sum_{m=1}^\infty \alpha_m^2 <\infty.\label{RMcond}\end{equation}
		The sequence $(\nu_{m}(\bs))_m$ generated by  \eqref{stoch_approx}  almost surely converges to a (possibly sample path dependent) compact connected internally chain transitive invariant set of \eqref{ODE}.
	\end{theorem}
	\begin{proof}
		We verify that assumptions (A1 - A4) in   \citet[Chapter 2]{borkar2009stochastic} are satisfied, from which the result follows by \citet[Chapter 2, Theorem 2]{borkar2009stochastic}:
		\begin{itemize}
			\item[(A1)] We have from Lemma~\ref{lemma:properties}  that $\tilde{f}_{\bs}$ is Lipschitz continuous.
			\item[(A2)] The Robbins-Monro conditions hold by assumption.
			\item[(A3)] The sequence $(\epsilon_m)_m$ defined by \eqref{MartDiffSeq} 
			is a bounded martingale difference sequence.
			\item[(A4)] Note that condition \eqref{RMcond} is sufficient for the condition on the step-size sequence in Lemma~\ref{lemma_nubound}, hence we have $\sup_m |\nu_{m}(\bs)|<\infty$ almost surely.   
		\end{itemize}
	\end{proof}
	The next corollary follows.
	\begin{corollary}\label{cor:convg_nu}
		Under Assumption~\eqref{RMcond}, the sequence $(\nu_{m}(\bs))_m$ generated by \eqref{stoch_approx}  converges almost surely to a random variable $\tilde{\nu}(\bs)\in[R_\ell, R_u]$ such that 
		$$\tilde{f}_{\bs}\left(\tilde{\nu}(\bs)\right) = 0.$$
	\end{corollary}
	\begin{proof}
		From the proof of Lemma~\ref{lemma_nubound}, we have $f_{\bs}(\nu) < 0$ for $\nu< R_\ell$ and $f_{\bs}(\nu ) > 0$ for $\nu> R_u$.
		Hence, $  \frac{d}{dx}\nu(x,\bs)> 0$ if $\nu< R_\ell$ and $ \frac{d}{dx}\nu(x,\bs)< 0$ if $\nu> R_u$. From Lipschitz continuity of $f_\bs$ it then follows that no solution of \eqref{ODE} goes to infinity.
		Hence, by classification of solutions to a first-order scalar autonomous ODE, we know that each solution of \eqref{ODE} must converge to a (semi-)stable point contained in the set $\mathcal{R}_{\bs}$ of roots of $f_{\bs}$, which is non-empty and contained in $[R_\ell, R_u]$ by the above discussion.
	\end{proof}
	If the step-size sequence is constant, i.e., $\alpha_m \equiv \alpha\in(0,\infty)$ for all $m$, the following result holds by \citet[Chapter 9, Theorem 3]{borkar2009stochastic}, which states that $\nu_{m}(\bs)$ is close to $\mathcal{R}_{\bs}$ in the limit in mean-square, but does not necessarily almost surely converge to a point in $\mathcal{R}_\bs$.
	\begin{theorem}[\sc Convergence of stochastic approximation for constant step-size]\label{Thm:constantSZ}
		\phantom{.}\\For a constant $H>0$
		\begin{equation*}
			\underset{m\rightarrow\infty}{\lim\sup}\;\mathbb{E}\left[\min_{\nu\in \mathcal{R}_{\bs}}\left(\nu_{m}(\bs) - {\nu}\right)^2\right] \leq H\alpha.
		\end{equation*}
	\end{theorem}
	\begin{proof}
		We already saw in the proof of Theorem~\ref{convgRM} that assumptions (A1) and (A3) \mbox{in  \citet[Chapter~2]{borkar2009stochastic}} are satisfied. Furthermore (9.2.1) and (9.2.2) in~\citet{borkar2009stochastic} are satisfied as the iterates $(\nu_{m}(\bs))_m$ stay in a closed bounded interval (Lemma~\ref{lemma_nubound}). 
	\end{proof}
	\begin{remark}
		In the above results, we have shown two convergence results under two different assumptions on the step-size sequence. Theorem~\ref{convgRM} assumes that the Robbins-Monro conditions almost surely hold for the step-size sequence $(\alpha_m)_m$. This is a stronger condition than the one assumed in Section \ref{sec:3.3}, where only $\sup_m\alpha_m<M_{\alpha}$ was assumed for some bound $M_{\alpha}$.    Under this stronger condition we were able to show that the stochastic process $(\nu_m)_m$ converges to a root of the function $\tilde{f}_\bs$ and based on a sample path alone, we can determine whether the sequence has converged or not. If we instead take a constant step-size sequence, we could only show that the limit of the stochastic process is close to $\mathcal{R}_\bs$, in terms of mean-squared difference. Often the rate of convergence is much higher for constant step-size~sequences~(\citet{borkar2009stochastic}). 
	\end{remark}
	The following lemma gives a recursive formula for the derivative which can be used for selection of the step-size in the adaptive stochastic approximation procedure \eqref{stoch_approx}.
	\begin{lemma}\label{Thm:deriv}
		\Thmderiv
	\end{lemma}
	\begin{proof}
		The full proof can be found in Appendix \ref{proofs}. The result might seem straightforward at first but the situation is more difficult due to the dependence of $U^{(k)}_{\bi,n}$ on $\nu$.
		The result is obtained by showing continuity of $U^{(k)}_{\bi,n}$ for all but at most finitely many points, by giving a coinciding lower and upper bound for the derivative of ${\arg\min}_{u\in[N]}Z_{\bi,u,n}^{(k)}$ whenever it exists and then using dominated convergence to show the derivative for the expected sum truncated to~$[-1/2,1/2]$.
	\end{proof}
	The next theorem states that using Lemma~\ref{Thm:deriv} an adaptive stochastic approximation method can be designed with asymptotically optimal variance (\citet{lai1979adaptive}). 
	
	\begin{theorem}[\sc Central limit theorem for stochastic approximation iterates]\label{thm:convg_dist}  \phantom{.}\\
		\Thmconvgdist
	\end{theorem}
	\begin{proof}
		
		The full proof can be found in Appendix \ref{proofs}. We first show that the sequence $m\alpha_m$ converges to a constant, by showing that $\alpha_m$ satisfies the Robbins-Monro conditions almost surely and by applying Corollary~\ref{cor:convg_nu}. After this, we can apply the results in \citet{lai1978limit}, where we have to account for the fact that the residuals are not independent and identically distributed but bounded, and the function is only differentiable at all but finitely many points.
	\end{proof}
	\begin{remark}[Step-size sequence and stopping criterion]\label{rem:criterion}We propose to use the adaptive step-size sequence $\alpha_m~=~1/|\sum_{\ell=1}^m h_\ell(\nu_\ell(\bs))|$ for the stochastic approximation sequence \eqref{stoch_approx}, and base the stopping criterion for the stochastic approximation sequence on the estimated radius of the confidence interval implied by \eqref{eqn:CLT_SA}.
		We stop  the stochastic approximation procedure \eqref{stoch_approx} when the estimated confidence radius based on Theorem~\ref{thm:convg_dist} is small enough, i.e., when \begin{equation}
			\frac{ C_{1-\beta}\sqrt{\frac{1}{m}\sum_{\ell = 1}^m(V_{\bs,\ell}(\nu_{\ell}(\bs)))^2}}{\sqrt{m}|\frac{1}{m}\sum_{\ell=1}^m h_{\ell}(\nu_{\ell}(\bs))|}\leq \epsilon_\nu ,\label{stoppingcriterium}
		\end{equation}
		with $C_{1-\beta}$ the level $1-\beta/2$ quantile of the standard normal distribution, and $\epsilon_\nu$ a pre-specified tolerance.
	\end{remark}
	\begin{remark}\label{rem:behavior_large_n}
		When all values $n(\bi,k)$ are large enough, we have by Theorem~\ref{theorem:approxthm}
		\begin{align}\tilde{f}_\bs(\nu)&\approx \sum_{k=1}^K\mathbb{E}_\bs\left[\min_{u\in [N]} Z_u^{(k)}(\nu)\right]\leq \inf_{\tau\in\mathcal{T}_N}\mathbb{E}_\bs[Z_\tau(\nu)]=f_{\bs}(\nu),\label{firstineq_remark}\\
			f_{\bs}(\nu)&\leq \sum_{k=1}^K\mathbb{E}_\bs\left[\min_{u\in [N]} Z_u^{(k)}(\nu)\right] + 1/(K+1) \approx \tilde{f}_\bs(\nu) + 1/(K+1)\label{secondineq_remark},
		\end{align}
		where we used $b=-a=1/2.$
		As both $\tilde{f}_\bs$ and $f_\bs$ are continuous, negative at $R_\ell$, positive at $C_R$, and $f_\bs$ is increasing,  we expect from \eqref{firstineq_remark} that the root of $\tilde{f}_\bs$ is an upper bound of the root of ${f}_\bs$. As $\mathbb{E}_\bs[\min_{u\in[N]}Z^{(k)}_u(\nu)]\geq 0$ for $k\geq 2$, we furthermore expect that the root of $\tilde{f}_\bs$ decreases to the root of $f_\bs$, and by \eqref{secondineq_remark} we expect that both roots coincide in the limit as $K\rightarrow\infty$. We furthermore have by \eqref{Ineq: bounds_f} and \eqref{secondineq_remark} that, for large enough values of $n(\bi,k)$,
		\begin{equation}
			\tilde{\nu}(\bs) - \nu_\sigma(\bs)\leq  f_\bs(\tilde{\nu}(\bs))/c\leq 1/(c\cdot(K+1)),\label{convg_root}
		\end{equation}
		hence the difference in the two roots is of order $O(1/(K+1))$.
		In order to get an accurate SBGIA, we hence propose, for a fixed $K$, to first increase all $n(\bi,k)$ to large enough values such that the root of $\tilde{f}_\bs$ has converged. After convergence has occurred, we propose to increase $K$ and to repeat this procedure until the root has also converged in~$K$.  
	\end{remark}
	
	\section{Application to Bayesian multi-armed bandits}\label{sect:numerical_results}
	
	We introduce the Bayesian \mbox{multi-armed} bandit in  Section \ref{sec:intro_bayes_experiment} as an application of the FABP. In Section \ref{sec:52}, we  consider outcome distributions from an exponential family. Subsequently, we present results for the SBGIA applied to two Bayesian bandits known from literature, the Bernoulli bandit and Gaussian bandit with known variance. In Section \ref{Subsec:Randomeffects}, we evaluate the performance of the SBGIAP for a novel Gaussian random effects bandit problem.
	
	\subsection{Bayesian multi-armed bandit}\label{sec:intro_bayes_experiment}
	
	Consider $A$ distributions with support $\mathcal{O}$. Selecting distribution $a$ at time $t\in \mathbb{N}$, results in a realisation of the random variable $O_t^a$ that has a known density $p(O_t^a|\btheta_a)$
	w.r.t. a measure $\mu$, where the unknown parameter $\btheta_a$ lies in a parameter space $\Theta$ (shared for all $a$). The random variables $O_t^a$ are assumed independent. 
	We perform a Bayesian analysis, where the parameters $\btheta_a$ are independent a priori and endowed with prior probability measure $\Pi^a_0$ w.r.t.\ the Polish space $(\Theta, \mathcal{B}(\Theta))$.
	Given this probability measure $\Pi_0^a$ on $(\Theta, \mathcal{B}(\Theta))$, we can determine the predictive distribution  \begin{equation}p(O\mid\Pi_0^a) = \int_{\Theta}p(O\mid\btheta) d\Pi_0^a(\btheta).\label{postupdating}\end{equation}
	A sample from this distribution then, in turn, generates a posterior distribution $\Pi_1^a(\cdot \mid O,\Pi^a_0)$ by Bayes' rule, i.e.,
	\begin{equation}\Pi^a_1(E \mid O,\Pi^a_0) = 
		\int_E 
		\frac{p(O\mid\btheta)}{\int_\bTheta p(O\mid\btheta)d\Pi_0^a(\btheta)}d\Pi_0^a(\btheta)\;\;\;\; \forall E\in\mathcal{B}(\Theta).\label{posterior_update}\end{equation}
	We can use the predictive distribution \eqref{postupdating} and posterior updating rule \eqref{posterior_update} to determine $A$ Markov chains (arms), with states $(\Pi_t^a)_{a=1}^A$ corresponding to (posterior) distributions on $\Theta$. The state space of each Markov chain is the space $\mathbb{M}$ of probability measures on $(\Theta,\mathcal{B}(\Theta))$.
	Furthermore, we endow this state space with the sigma algebra $\mathcal{M}$, which is the smallest sigma field making all maps from $\mathbb{M}$ to $\mathbb{R}$ measurable. Then $(\mathbb{M},\,\mathcal{M})$ is a Borel space (\citet[Chapter 3.1]{ghosal2017fundamentals}). 
	Each Markov chain $(\Pi_t^a)_{a=1}^A$ has transition kernel
	\begin{equation}
		\mathbb{P}_{\Pi^a_t}(E) = \int_{\mathcal{O}} \mathbb{I}\left(\Pi_{t+1}^a(\cdot\mid O,\Pi^a_t)\in E\right)p(O\mid\Pi^a_t)\mu(dO) \;\;\;\; \forall E\in\mathcal{M}.\label{transkern:posterior}
	\end{equation}
	
	Our goal is to find a (Markov) policy $\pi$ to sequentially sample from one of the $A$ arms that maximizes the expected discounted sum of outcomes under the Bayesian model, i.e, to maximize
	\begin{equation}\mathbb{E}_{\pi}\left[\sum_{t=0}^\infty \gamma^t O_{t+1}^{A_t}\right] = \mathbb{E}_{\pi}\left[\sum_{t=0}^\infty \gamma^t R(\Pi_t^{A_t})\right],\label{eqn:Totreward_omschrijven} \end{equation}
	where $ R(\Pi^a_t)= \int_{\Theta} \mathbb{E}[O\mid\btheta]d\Pi^a_t(\btheta)$ is the posterior mean outcome for the current posterior $\Pi_t^a$. The equality in \eqref{eqn:Totreward_omschrijven} follows from Assumption \ref{boundedtotrew} and Fubini's theorem. 
	Letting $\bS_t^a = \Pi_t^a$ for all $a,t$, $\mathcal{S} = \mathbb{M}$, and $\mathcal{G} = \mathcal{M}$, we are in the setting of Section \ref{FABP} with transition kernel \eqref{transkern:posterior} and reward function $R$.

	
	
	\subsection{Gittins index approximation results}\label{sec:52}
	This section considers the FABP  with distributions  from an exponential family as detailed in Section \ref{ExpFam}. Specific results are presented for Bernoulli and Gaussian families in Sections \ref{sec:results_Bernoulli} and \ref{sec:results_Gaussian}.

	For determining the SBGIA, we set $\beta = 0.05$ and $\epsilon_\nu = 0.001$ in \eqref{stoppingcriterium} to estimate a $95\%$ asymptotic confidence interval for $\tilde{\nu}(\bs)$ with radius $0.001$.
	We compare the SBGIA with the Calibration method introduced in~\citet{gittins1979bandit}, which is a combination of a bisection method and backward induction to obtain the value of the truncated optimal stopping problem in \eqref{defGI2}. The parameters of the Calibration method are set such that the approximation error is very small, so that we may consider these values to be the true Gittins index values.
	
	\subsubsection{Exponential families~\label{ExpFam}}

	Assume the data comes from a distribution belonging to an exponential family,  i.e., for known functions $\bpsi,\eta,\rho,\zeta$ we have
	$$p(O^a\mid\btheta_a) = \zeta(O^a)\exp(\eta(\btheta_a)^\top \bpsi(O^a)-\rho(\btheta_a))d\mu(O^a),$$
	and
	$$p((O^a_t)_t\mid\btheta_a) = \prod_{u=1}^t\zeta(O^a_u)\exp\left(\eta(\btheta_a)^\top \bpsi(O^a_{u})-\rho(\btheta_a)\right)d\mu(O^a_{u}).$$
	Now assume  a conjugate prior for this model (\citet{diaconis1979conjugate}), with normalizing constant $\zeta_2$
	$$p(\btheta_a;\bPsi_0^a,\kappa^a_0) = \zeta_2(\bPsi^a_0,\kappa^a_0) \exp(\eta(\btheta_a)^\top\bPsi^a_0 - \kappa^a_0 \rho(\btheta_a)).$$ Letting $\bPsi_t^a = \bPsi^a_0 + \sum_{u=1}^t\bpsi(O_{u}^a)$ and $\kappa^a_t = \kappa^a_0+t$, we have the following expression for the posterior
	\begin{align*}
		p(\btheta_a\mid(O^a_t)_t)&=\zeta_2(\bPsi^a_t,\kappa^a_t) \exp(\eta(\btheta_a)^\top\bPsi^a_t - \kappa^a_t \rho(\btheta_a))=:p(\btheta_a;\bPsi^a_t,\kappa^a_t).
	\end{align*}
	The vector $\bPsi_t^a$ is often referred to as the sufficient statistic, and $\kappa^a_t$ as the effective number of observations, which is the sum of a prior number of observations $\kappa^a_0$ and the actual number of observations $t$ for arm $a.$
	As the only random element in the posterior is $\bPsi^a_t$, which we will assume to lie in $\mathbb{R}^d$, the Markov chain $(\Pi^a_t)_t$ can be represented by the time-inhomogeneous Markov chain $(\bPsi^a_t,\kappa^a_t)_t$ on the finite-dimensional state space  $\mathbb{R}^{d+1}$ with transition dynamics
	\begin{equation}
		(\bPsi^a_t,\kappa^a_t)\rightarrow (\bPsi^a_t, \kappa^a_t) + (\bpsi(O),1),\;\;\;\; O\sim p(O\mid(\bPsi^a_t,\kappa^a_t)).\label{transition_dynamics_suff}
	\end{equation}
	The  rewards  $R(\bPsi^a_t,\kappa^a_t) = \int_\Theta \mathbb{E}[O\mid \btheta] \;p(\btheta;\bPsi^a_t,\kappa^a_t)d\btheta$ are the posterior mean outcomes. As our goal is to maximize the expected discounted sum of outcomes under the Bayesian model, letting \mbox{$\bS_t^a = (\bPsi_t^a, \kappa^a_t)$} for all $a,t$, $\mathcal{S} = \mathbb{R}^{d+1}$, and $\mathcal{G} = \mathcal{B}(\mathbb{R}^{d+1})$, we are in the setting of Section \ref{FABP} with the transition kernel implied by \eqref{transition_dynamics_suff} and reward function $R$. 
	


\subsubsection{Bernoulli bandit \label{sec:results_Bernoulli}}
We consider the case when $O^a_{t}$ are Bernoulli distributed with unknown success probability $p_a\in[0,1]$, i.e., 
$$\mathbb{P}(O^a = 1) = p_a,\;\;\;\; \mathbb{P}(O^a = 0) = 1-p_a,$$ which implies that the family of outcome distributions is an exponential family with 
\begin{align*}
	\zeta&\equiv 1,\;\eta(p_a)=\log(p_a/(1-p_a)),\;\Psi(O) = O,\;\rho(p_a) = -\log(1-p_a),
\end{align*} and $\mu$ the counting measure on the nonnegative integers.
We assume a conjugate $\text{Beta}(\alpha_a,\beta_a)$ prior on each $p_a$, hence $\Psi^a_t = \alpha_a + \sum_{u=1}^tO^a_u$ and  $\kappa^a_t = \alpha_a + \beta_a + t$.
In the following, we drop the superscript $a$ from notation as we consider results for a single arm $a$ only.

Observe that $R(\Psi,\kappa) = \Psi/\kappa$, hence for a fixed horizon $N\in\mathbb{N}$ the Markov reward process starting from $(\Psi, \kappa)$ is bounded. Thus, in Section \ref{sect: theoretical_results} we can take $\sigma = N$, and the result of Theorem~\ref{boundtrunc} holds without truncation of the support of the rewards. Using $\nu(\bS_{N})^+\leq 1$, and letting the Gittins index approximation after truncation be denoted by $\nu_N(\Psi,\kappa)$ (as $
\sigma = N$), we have
\begin{equation}
	0\leq \nu(\Psi,\kappa)-\nu_N(\Psi,\kappa)\leq \frac{\gamma^N}{1-\gamma^N},\label{bound_TE}
\end{equation}
hence for $\epsilon_\text{trunc}>0$
\begin{equation}
	N\geq \log_\gamma(\epsilon_\text{trunc}/(1+\epsilon_\text{trunc}))\implies \nu(\Psi,\kappa)-\nu_N(\Psi,\kappa)\leq \epsilon_\text{trunc}.\label{requirement_N_bernoulli}
\end{equation}
For the rewards, we have $$R(\Psi, \kappa+w)\in [ \Psi/(\kappa+ w),\; (\Psi+w)/(\kappa+ w)]\;\;\;\; \forall w\in [N].$$
Starting from $(\Psi, \kappa)$ the rewards lie in the bounded interval $[R_\ell,\, R_u]$ up to a horizon $N$ for $$R_\ell = \Psi/(\kappa + N)\text{, and }R_u= (\Psi+N)/(\kappa+N),$$ which can be used in \eqref{Z_unn}.

Table  \ref{tab:res_bernoulli_gamma80_T21_exp_k} shows the Gittins index values found under the  SBGIA and the Calibration method for the Bernoulli bandit with $\gamma=0.8$.
For the SBGIA and the Calibration method, we set $N=35$, corresponding to a truncation error bound $\epsilon_{\text{trunc}}$ of $0.0005$ according to \eqref{requirement_N_bernoulli}.  We set each $n(\bi,k)=1$ and considered $K=1,2,3$. 
For ease of comparison, Table~\ref{tab:res_bernoulli_gamma80_T21_exp_k} shows  $\kappa-\Psi$, the effective number of failures, as the second state variable.
Table~\ref{tab:res_bernoulli_gamma80_T21_exp_k} shows the log computation time (in seconds), estimated bias (SBGIA minus Calibration), root-mean-squared error (RMSE), and average standard deviation over estimates for the SBGIA (calculated as the average deviation in the approximation over states for two independent simulation runs) for each considered value of $K$. Using \eqref{convg_root},  the rightmost column of Table~\ref{tab:res_bernoulli_gamma80_T21_exp_k} shows an estimate of the limit as $K\rightarrow\infty$, found by performing ordinary least squares on the columns $K=1,2,3$ with the line \mbox{$\nu(\bs)=a + b/(K+1)$}.

Table~\ref{tab:res_bernoulli_gamma80_T21_exp_k} shows that the Gittins index is overestimated by the SBGIA for $K=1$, which is in agreement with our expectation, as the expected minimum is always smaller than the minimum of the expectation over stopping times (see, e.g.,  \citet{chen2018beating}). When $K$ increases, the amount of overestimation decreases, and 
values of the SBGIA lie above the Gittins index computed by the Calibration method for any state and value of $K$.
It follows from Remark \ref{rem:behavior_large_n} that this behaviour should occur asymptotically when the values of $n(\bi, k)$ go to infinity;  the numerical results show that it also occurs for small values of $n(\bi,k)$.
The estimate of the limit when $K\rightarrow\infty$, shown in the rightmost column of Table~\ref{tab:res_bernoulli_gamma80_T21_exp_k} has lowest bias and RMSE, indicating a correct assumption on the $O(1/(K+1))$ convergence rate. The computation time increases more than tenfold with each increase in $K$, and for $K=1$ it is already about 100 times larger than the computation time of the Calibration method. The standard deviation over runs (SD) is quite low, around $0.0001$ for all values of $K$, indicating that the estimates are consistent over independent runs.

\begin{table}[h!]
	\centering
	\caption{Gittins index values found under the SBGIA for $K=1,2,3$, the Calibration method, and an estimate for the limit as $K\rightarrow\infty$ for the Bernoulli bandit.  We set each $n(\bi,k)=1$, $\gamma = 0.8, N = 35$, and set a tolerance of $\epsilon_\nu=0.001$ in the stopping criterion \eqref{stoppingcriterium}  for determining the SBGIA. The computation (CPU) time denotes the time it took to calculate all values in the column and is measured in seconds. The bias, RMSE, and standard deviation (SD) are multiplied by $100$, i.e., displayed in percentage points. \label{tab:res_bernoulli_gamma80_T21_exp_k}}
	\begin{tabular}{lccccccc}\hline
		General & $ \Psi$ & $ \kappa - \Psi$ & Calibration & $ K=1$ & $ K=2$ & $ K=3$ & Est. limit\\\hline
		& 1 & 1 & 0.641 & 0.643 & 0.642 & 0.642 & 0.641\\
		& 1 & 2 & 0.443 & 0.447 & 0.446 & 0.445 & 0.442\\
		& 1 & 3 & 0.332 & 0.338 & 0.337 & 0.335 & 0.332\\
		& 1 & 4 & 0.263 & 0.270 & 0.268 & 0.267 & 0.264\\
		& 1 & 5 & 0.216 & 0.224 & 0.222 & 0.221 & 0.218\\
		& 1 & 6 & 0.183 & 0.191 & 0.189 & 0.188 & 0.185\\
		& 2 & 1 & 0.760 & 0.760 & 0.760 & 0.760 & 0.759\\
		& 2 & 2 & 0.590 & 0.592 & 0.591 & 0.591 & 0.590\\
		& 2 & 3 & 0.476 & 0.480 & 0.479 & 0.478 & 0.476\\
		& 2 & 4 & 0.398 & 0.402 & 0.401 & 0.400 & 0.398\\
		& 2 & 5 & 0.340 & 0.345 & 0.344 & 0.342 & 0.340\\
		& 2 & 6 & 0.296 & 0.301 & 0.300 & 0.299 & 0.297\\
		& 3 & 1 & 0.816 & 0.816 & 0.816 & 0.816 & 0.815\\
		& 3 & 2 & 0.671 & 0.673 & 0.673 & 0.672 & 0.671\\
		& 3 & 3 & 0.566 & 0.568 & 0.568 & 0.567 & 0.566\\
		& 3 & 4 & 0.487 & 0.490 & 0.489 & 0.489 & 0.487\\
		& 3 & 5 & 0.427 & 0.430 & 0.429 & 0.428 & 0.427\\
		& 3 & 6 & 0.379 & 0.383 & 0.382 & 0.381 & 0.379\\
		& 4 & 1 & 0.849 & 0.850 & 0.849 & 0.849 & 0.849\\
		& 4 & 2 & 0.725 & 0.726 & 0.725 & 0.725 & 0.724\\
		& 4 & 3 & 0.628 & 0.629 & 0.629 & 0.629 & 0.628\\
		& 4 & 4 & 0.552 & 0.554 & 0.554 & 0.553 & 0.552\\
		& 4 & 5 & 0.491 & 0.494 & 0.494 & 0.493 & 0.492\\
		& 4 & 6 & 0.443 & 0.446 & 0.445 & 0.444 & 0.443\\
		& 5 & 1 & 0.872 & 0.873 & 0.872 & 0.872 & 0.871\\
		& 5 & 2 & 0.762 & 0.763 & 0.763 & 0.762 & 0.762\\
		& 5 & 3 & 0.674 & 0.675 & 0.675 & 0.674 & 0.673\\
		& 5 & 4 & 0.602 & 0.604 & 0.603 & 0.603 & 0.602\\
		& 5 & 5 & 0.543 & 0.545 & 0.545 & 0.544 & 0.543\\
		& 5 & 6 & 0.494 & 0.497 & 0.496 & 0.496 & 0.495\\
		& 6 & 1 & 0.888 & 0.889 & 0.888 & 0.888 & 0.887\\
		& 6 & 2 & 0.790 & 0.791 & 0.791 & 0.791 & 0.790\\
		& 6 & 3 & 0.709 & 0.710 & 0.710 & 0.709 & 0.709\\
		& 6 & 4 & 0.641 & 0.643 & 0.642 & 0.642 & 0.641\\
		& 6 & 5 & 0.585 & 0.586 & 0.586 & 0.586 & 0.585\\
		& 6 & 6 & 0.537 & 0.539 & 0.538 & 0.538 & 0.537\\\hline
		$\log_{10}$ CPUtime &  &  & 0.132 & 2.650 & 3.810 & 4.810 & \\
		Bias (x0.01) &  &  &  & 0.277 & 0.210 & 0.144 & 0.008\\
		RMSE (x0.01) &  &  &  & 0.342 & 0.264 & 0.189 & 0.071\\
		SD (x0.01) &  &  &  & 0.011 & 0.010 & 0.014 & \\\hline
	\end{tabular}
\end{table}

Table~\ref{tab:res_bernoulli_gamma80_T21_exp_n} shows the values of the SBGIA and the Calibration method for $K=2$ and $n(\bi,2) = 1,3,5$, i.e., only the nested number of simulations for the SBGIA is increased. The CPU time, bias, RMSE, and standard deviations over two independent simulation runs are shown at the bottom of Table~\ref{tab:res_bernoulli_gamma80_T21_exp_n}.

Table~\ref{tab:res_bernoulli_gamma80_T21_exp_n} shows that when increasing $n(\bi,k)$ a smaller error (expressed in bias and RMSE) is attained for a lower computational cost in comparison to Table~\ref{tab:res_bernoulli_gamma80_T21_exp_k}. 
The results in Table~\ref{tab:res_bernoulli_gamma80_T21_exp_n} hence agree with the proposal in Remark \ref{rem:behavior_large_n}, as increasing the number of nested simulations leads to higher precision in less computation time.

\begin{table}[h!]
	\centering
	\caption{Gittins index values found under the SBGIA for $n:= n(\bi,2) = 1,3,5$ and  the Calibration method for the Bernoulli bandit.  We set $K=2$, $\gamma = 0.8, N = 35$, and set a tolerance of $\epsilon_\nu=0.001$ in the stopping criterion \eqref{stoppingcriterium}  for determining the SBGIA. The computation (CPU) time denotes the time it took to calculate all values in the column and is measured in seconds. The bias, RMSE, and standard deviation (SD) are multiplied by $100$, i.e., displayed in percentage points. \label{tab:res_bernoulli_gamma80_T21_exp_n}}
	\begin{tabular}{lcccccr}\hline
		General & $ \Psi$ & $ \kappa - \Psi$ & Calibration & $ n=1$ & $ n=3$ & $ n=5$\\\hline
		& 1 & 1 & 0.641 & 0.642 & 0.642 & 0.642\\
		& 1 & 2 & 0.443 & 0.446 & 0.445 & 0.444\\
		& 1 & 3 & 0.332 & 0.337 & 0.334 & 0.333\\
		& 1 & 4 & 0.263 & 0.268 & 0.266 & 0.265\\
		& 1 & 5 & 0.216 & 0.222 & 0.220 & 0.219\\
		& 1 & 6 & 0.183 & 0.189 & 0.187 & 0.186\\
		& 2 & 1 & 0.760 & 0.760 & 0.760 & 0.760\\
		& 2 & 2 & 0.590 & 0.591 & 0.591 & 0.590\\
		& 2 & 3 & 0.476 & 0.479 & 0.478 & 0.477\\
		& 2 & 4 & 0.398 & 0.401 & 0.399 & 0.399\\
		& 2 & 5 & 0.340 & 0.344 & 0.342 & 0.341\\
		& 2 & 6 & 0.296 & 0.300 & 0.299 & 0.298\\
		& 3 & 1 & 0.816 & 0.816 & 0.816 & 0.816\\
		& 3 & 2 & 0.671 & 0.673 & 0.672 & 0.672\\
		& 3 & 3 & 0.566 & 0.568 & 0.567 & 0.566\\
		& 3 & 4 & 0.487 & 0.489 & 0.489 & 0.488\\
		& 3 & 5 & 0.427 & 0.429 & 0.428 & 0.428\\
		& 3 & 6 & 0.379 & 0.382 & 0.380 & 0.380\\
		& 4 & 1 & 0.849 & 0.849 & 0.849 & 0.849\\
		& 4 & 2 & 0.725 & 0.725 & 0.725 & 0.725\\
		& 4 & 3 & 0.628 & 0.629 & 0.629 & 0.628\\
		& 4 & 4 & 0.552 & 0.554 & 0.553 & 0.553\\
		& 4 & 5 & 0.491 & 0.494 & 0.493 & 0.493\\
		& 4 & 6 & 0.443 & 0.445 & 0.444 & 0.444\\
		& 5 & 1 & 0.872 & 0.872 & 0.872 & 0.872\\
		& 5 & 2 & 0.762 & 0.763 & 0.763 & 0.762\\
		& 5 & 3 & 0.674 & 0.675 & 0.674 & 0.674\\
		& 5 & 4 & 0.602 & 0.603 & 0.603 & 0.602\\
		& 5 & 5 & 0.543 & 0.545 & 0.544 & 0.544\\
		& 5 & 6 & 0.494 & 0.496 & 0.496 & 0.495\\
		& 6 & 1 & 0.888 & 0.888 & 0.888 & 0.888\\
		& 6 & 2 & 0.790 & 0.791 & 0.791 & 0.790\\
		& 6 & 3 & 0.709 & 0.710 & 0.709 & 0.709\\
		& 6 & 4 & 0.641 & 0.642 & 0.642 & 0.642\\
		& 6 & 5 & 0.585 & 0.586 & 0.585 & 0.585\\
		& 6 & 6 & 0.537 & 0.538 & 0.538 & 0.537\\\hline
		$\log_{10}$ CPUtime &  &  & 0.132 & 3.810 & 3.950 & 4.170\\
		Bias (x0.01) &  &  &  & 0.210 & 0.113 & 0.059\\
		RMSE (x0.01) &  &  &  & 0.264 & 0.145 & 0.085\\
		SD (x0.01) &  &  &  & 0.010 & 0.016 & 0.013\\\hline
	\end{tabular}
\end{table}

\FloatBarrier
\begin{remark}
	The length of the confidence interval for the Gittins index is in large part determined by the error bound for the optimal stopping value approximation given in~\citet{chen2018beating}. The contribution of the optimal stopping approximation is equal to  $2B(\delta,\xi)/c$~(Theorem~\ref{FTCI}), which can be seen as the bias of approximating $\nu_{\sigma}(\bs)$ by $\nu_M(\bs)$. Here $c$ is the slope of $Z_t(\nu)$ in $\nu$ and $B(\delta,\xi)$ is a bound for the error induced by the approximation method introduced in~\citet{chen2018beating}. 
	The bound is similar to the radius of the confidence interval found under the Delta method when applied to sampled approximations to the truncated Gittins index $f_\bs^{-1}(0)=\nu_\sigma(\bs)$, which would be proportional to $1/f_{\bs}'(0)$. This implies that the bound could be sharp, given that the error bound $B(\delta,\xi)$ is sharp.
	Note that rescaling $Z$ in \eqref{Z_unn} would not alter this radius, increasing the range of $Z$ linearly increases the error according to Theorem~\ref{theorem:approxthm}, yielding the same confidence radius. 
	The results in Section~\ref{sec:52} indicate that the bound can be made tighter. For instance, for $\Psi=1$, $\kappa=2$, and $K=2$, an absolute bias of \hbox{$0.642 - 0.641 = 0.001$} is seen in Table~\ref{tab:res_bernoulli_gamma80_T21_exp_k}. We have $R_u-R_\ell=36/37-1/37=0.946$, \hbox{$c = (1-\gamma)/(2(R_u-R_\ell)(1-\gamma^N))= 0.106$}, hence to get the theoretical bound for the bias less than~0.001, we should at least have \hbox{$B(\delta,\xi)=\xi+\delta\leq 0.106\cdot 0.001/2=5.29\cdot 10^{-5}$.} Setting~$\xi=\delta=5.29\cdot 10^{-5}/2$, following \citet{chen2018beating}, we would need $K=\floor{1/\xi} =378\cdot 10^2$ and $n(\bi,k) = \ceil{\log(2/\delta)/(2\xi^2)} = 8.04\cdot 10^9$ to obtain a bias of $0.001$ for $\nu_\sigma(\bs).$
	The main limiting factor in applying the theoretical bound in practice is hence the bound from~\citet{chen2018beating} which could possibly be made more tight.
\end{remark}
\FloatBarrier

\subsubsection{Gaussian bandit \label{sec:results_Gaussian}}
Let $O^a_{t}$ be normally  distributed with unknown mean $\theta_a$, and known variance for each arm $a$. By scaling the mean and outcomes  by the (known) standard deviation, we can equivalently assume $O^a_{t}\sim \mathcal{N}(\theta_a,1)$, which implies that the family of outcome distributions is an exponential family with, letting $\phi$ denote the standard normal density,
\begin{align}
	\zeta&\equiv \phi,\;\eta(\theta_a) = \theta_a,\; \bpsi(O)=O ,\; \rho(\theta_a)=\theta_a^2/2,\label{eqn:specificationexpfam_gauss}
\end{align}
and $\mu$ the Lebesgue measure. 
We assume a $\mathcal{N}(\mu_0,v_a)$ prior on $\theta_a$, hence \mbox{$\kappa^a_t = 1/v_a + t$}, \mbox{$\Psi^a_t = \mu_0/v_a + \sum_{u=1}^tO_u^a$.}
In the following, we drop the superscript $a$ from notation as we consider results for a single arm $a$ only.

Observe from \eqref{eqn:specificationexpfam_gauss} that $O\mid (\Psi_t, \kappa_t)\sim \mathcal{N}(\Psi_t/\kappa_t, \,1+ 1/\kappa_t)$, hence $$R(\Psi_t,\kappa_t) = \mathbb{E}[O\mid \Psi_t,\kappa_t]=\Psi_t/\kappa_t,$$ and
\begin{equation}
	R(\Psi_{t+1},\kappa_{t+1})\mid (R(\Psi_t,\kappa_t),\kappa_t)  \sim \mathcal{N}\left(R(\Psi_t,\kappa_t),\; 1/(\kappa_t\kappa_{t+1})\right).\label{MCdef}
\end{equation}
From \citet{yao2006some} it holds that
$$ \nu(\Psi_t,\kappa_t) = R(\Psi_t,\kappa_t) +  \nu(0, \kappa_t).$$
It is hence sufficient to calculate the Gittins index for Gaussian rewards given that the initial sufficient statistic is zero.
Observe from \eqref{MCdef} that, starting from the initial state $(0,\kappa_t)$, the process $(R(\Psi_{t}, \kappa_{t+u}))_u$ is a Gaussian random walk starting at zero with normally distributed, zero-centered increments with \mbox{variance $1/((\kappa_t + u-1)(\kappa_t + u))$.}

Let the stopping time $\sigma$ be defined as in \eqref{def:stoppingtime} for fixed $N\in\mathbb{N}$, and let $R_\ell=-L$, $R_u=L$ for \mbox{fixed $L>0.$} We then have by Kolmogorov's inequality  \begin{align*}&\mathbb{P}_{(0,\kappa_t)}(\sigma < N)\leq  \frac{\mathbb{E}[R(\Psi_{N-1},\kappa_{t+N-1})^2]}{L^2} =\frac{1}{L^2}\sum_{u=1}^{N-1}1/(\kappa_{t+u-1}\kappa_{t+u})\leq \frac{1}{L^2\kappa_t}.\end{align*}
We hence have $\mathbb{E}_{(0,\kappa_t)}[\gamma^{\sigma}] \leq \gamma^N + \frac{1}{L^2\kappa_t}$.
Note that $$C(0,\kappa_t)= \sum_{u=0}^\infty\gamma^t \mathbb{E}_{(0,\kappa_t)}|R(\Psi_{u},\kappa_{t+u})| = \sum_{u=1}^\infty\gamma^u \sqrt{\frac{2}{\pi}\left(\frac{1}{\kappa_t } - \frac{1}{\kappa_{t+u}}\right)}\leq \frac{\sqrt{2/\kappa_t }}{1-\gamma}. $$
From Theorem~\ref{boundtrunc}, and as $\nu(\bS_\sigma)^+\leq (1-\gamma)\mathbb{E}[C(0,\kappa_\sigma)]\leq \sqrt{2}$, it follows that 
\begin{equation}
	\nu(0,\kappa_t) - \nu_{\sigma}(0,\kappa_t) \leq \frac{\sqrt{2}\,\mathbb{E}_{(0,\kappa_t)}[\gamma^{\sigma}]}{1-\mathbb{E}_{(0,\kappa_t)}[\gamma^{\sigma}]},\label{truncbound_Gaussian}
\end{equation}
The truncation error in \eqref{truncbound_Gaussian} is smaller than $\epsilon_\text{trunc}>0$ when, e.g.,
\begin{equation}N \geq \log_\gamma\left(\epsilon_\text{trunc}/(2(\sqrt{2}+\epsilon_\text{trunc}))\right)\text{ and }L\geq \sqrt{2(\sqrt{2} + \epsilon_\text{trunc})/(\kappa_t\epsilon_\text{trunc})}.\label{choice_L_N}\end{equation}
Values of $R_u = -R_\ell=L$ and $N$ such that the above inequalities are satisfied can then be used in~\eqref{Z_unn}.

Table~\ref{tab:res_gaussian_gamma80_T21}  shows the Gittins index values found  under the SBGIA  and the Calibration method for the Gaussian bandit with $\gamma=0.8$.
For the SBGIA, we set $N=39$, corresponding to a truncation error bound $\epsilon_\text{trunc}=0.0005$ according to \eqref{truncbound_Gaussian}. For all states $(0,\kappa_t)$ the value of $L$ was set to the lower bound in \eqref{choice_L_N}.
We next set each $n(\bi,k)=1$ and considered $K=1,2,3$.
The Gittins indices found under the Calibration method shown in Table~\ref{tab:res_gaussian_gamma80_T21} can also be derived from \hbox{\citet[Table 8.1]{gittins2011multi}.}
As we only consider $\Psi=0$, each state of the Gaussian bandit in Table~\ref{tab:res_gaussian_gamma80_T21} is denoted by~$\kappa$.
We show the log computation time at the bottom, as well as the  bias, RMSE, and standard deviation in percentage points. The rightmost column of the table shows an estimate of the limit as $K\rightarrow\infty$, found by performing ordinary least squares on the columns $K=1,2,3$ with the line \mbox{$\nu(\bs)=a + b/(K+1)$}.

Table~\ref{tab:res_gaussian_gamma80_T21} shows that, as in Table~\ref{tab:res_bernoulli_gamma80_T21_exp_k}, the Gittins index is overestimated for $K=1$, and the values for the SBGIA, as well as the error measures, decrease in~$K$. The estimates for the Gaussian bandit show larger errors than those for the Bernoulli bandit. Possibly due to the continuity in the support of the rewards, which induces a larger variance in the sampled paths. The computation times for the Gaussian bandit are also approximately ten times larger than those for the Bernoulli bandit.
The computation time shown in Table~\ref{tab:res_gaussian_gamma80_T21} is similar for $K=1,2$, and  increases tenfold when going from $K=2$ to $K=3$. The computation time of the Calibration method is comparable to the computation time for $K=1$, indicating that the Gaussian bandit with known variance is already a hard problem to solve under the Calibration method.
The low average standard deviation in Table~\ref{tab:res_gaussian_gamma80_T21} indicates that the estimates are consistent over different runs. The estimates of the limits in the rightmost column again show a better quality than those for finite $K$, often giving the value of the Gittins index with an error of 0.001 for $\kappa\geq 4.$

Table~\ref{tab:res_gaussian_gamma80_T21_varyingn} shows values of the SBGIA obtained when setting $K=2$ and varying $n(\bi,2)$ for the Gaussian bandit with unit variance. As in Table 
\ref{tab:res_bernoulli_gamma80_T21_exp_n}, it is seen that the errors decrease faster in $n(\bi,k)$ for $K$ fixed than vice versa, in agreement with the proposal in Remark \ref{rem:behavior_large_n}.


\begin{table}[h!]
	\centering
	\caption{Gittins index values found under the SBGIA for $K=1,2,3$, the Calibration method, and an estimate for the limit as $K\rightarrow\infty$  for the Gaussian bandit (unit variance). For columns $K=1,\dots, 3$ we set each $n(\bi,k)= 1$ $\gamma = 0.8,\, N = 39$, and set a tolerance of $\epsilon_\nu=0.001$ in the stopping criterion \eqref{stoppingcriterium}  for determining the SBGIA.
		The computation (CPU) time denotes the time it took to calculate all values in the column and is measured in seconds. 
		The bias, RMSE, and standard deviation (SD) are multiplied by $100$, i.e., displayed in percentage points.
		\label{tab:res_gaussian_gamma80_T21}}
	\begin{tabular}{lcccccc}\hline
		General & $\kappa$ & Calibration & $ K=1$ & $ K=2$ & $ K=3$ & Est. limit\\\hline
		& 1 & 0.505 & 0.526 & 0.520 & 0.520 & 0.513\\
		& 2 & 0.308 & 0.329 & 0.323 & 0.320 & 0.312\\
		& 3 & 0.226 & 0.245 & 0.239 & 0.237 & 0.229\\
		& 4 & 0.179 & 0.196 & 0.191 & 0.188 & 0.180\\
		& 5 & 0.149 & 0.164 & 0.160 & 0.157 & 0.150\\
		& 6 & 0.128 & 0.142 & 0.138 & 0.135 & 0.128\\
		& 7 & 0.112 & 0.125 & 0.121 & 0.119 & 0.113\\
		& 8 & 0.100 & 0.112 & 0.108 & 0.106 & 0.101\\
		& 9 & 0.090 & 0.101 & 0.098 & 0.096 & 0.091\\
		& 10 & 0.082 & 0.092 & 0.089 & 0.087 & 0.083\\
		& 20 & 0.043 & 0.050 & 0.048 & 0.047 & 0.044\\
		& 30 & 0.029 & 0.034 & 0.033 & 0.032 & 0.030\\
		& 40 & 0.022 & 0.026 & 0.025 & 0.025 & 0.023\\
		& 50 & 0.018 & 0.021 & 0.020 & 0.020 & 0.018\\\hline
		$\log_{10}$ CPUtime &  & 4.050 & 3.860 & 4.320 & 5.730 & \\
		Bias (x0.01) &  &  & 1.250 & 0.871 & 0.725 & 0.172\\
		RMSE (x0.01) &  &  & 1.380 & 0.970 & 0.818 & 0.263\\
		SD (x0.01) &  &  & 0.030 & 0.031 & 0.059 & \\\hline
	\end{tabular}
\end{table}

\begin{table}[h!]
	\centering
	\caption{Gittins index values found under the SBGIA for $n:= n(\bi,2) = 1,3,5$ and the Calibration method
		for the Gaussian bandit (unit variance). We set $K=2$, $\gamma = 0.8, N = 39$, and set a tolerance of $\epsilon_\nu=0.001$ in the stopping criterion \eqref{stoppingcriterium}  for determining the SBGIA. The computation (CPU) time denotes the time it took to calculate all values in the column and is measured in seconds. The bias, RMSE, and standard deviation (SD) are multiplied by $100$, i.e., displayed in percentage points.\label{tab:res_gaussian_gamma80_T21_varyingn}}
	\begin{tabular}{lccccc}\hline
		General & $ \kappa $ & Calibration & $ n=1$ & $ n=3$ & $ n=5$\\\hline
		& 1 & 0.505 & 0.520 & 0.514 & 0.512\\
		& 2 & 0.308 & 0.323 & 0.317 & 0.316\\
		& 3 & 0.226 & 0.239 & 0.235 & 0.232\\
		& 4 & 0.179 & 0.191 & 0.188 & 0.185\\
		& 5 & 0.149 & 0.160 & 0.157 & 0.155\\
		& 6 & 0.128 & 0.138 & 0.135 & 0.133\\
		& 7 & 0.112 & 0.121 & 0.119 & 0.116\\
		& 8 & 0.100 & 0.108 & 0.107 & 0.104\\
		& 9 & 0.090 & 0.098 & 0.096 & 0.094\\
		& 10 & 0.082 & 0.089 & 0.088 & 0.085\\
		& 20 & 0.043 & 0.048 & 0.047 & 0.046\\
		& 30 & 0.029 & 0.033 & 0.032 & 0.031\\
		& 40 & 0.022 & 0.025 & 0.025 & 0.024\\
		& 50 & 0.018 & 0.020 & 0.020 & 0.019\\\hline
		$\log_{10}$ CPUtime &  & 4.050 & 4.320 & 4.710 & 4.880\\
		Bias (x0.01) &  &  & 0.871 & 0.644 & 0.437\\
		RMSE (x0.01) &  &  & 0.970 & 0.698 & 0.481\\
		SD (x0.01) &  &  & 0.031 & 0.015 & 0.010\\\hline
	\end{tabular}
\end{table}

\FloatBarrier

\subsection{ Gaussian random effects bandit \label{Subsec:Randomeffects}}
This section compares the performance of a policy based on the SBGIA to that of policies Thompson sampling and Bayes-UCB in case each arm describes the posterior under a Gaussian random effects model. This multi-armed bandit model was not found in literature. 


In the Gaussian random effects bandit model, it is assumed that there an additional factor that induces heterogeneity within each of the $A$ distributions of choice. The  factor induces multiple clusters to which the outcomes are assigned. 
Outcomes assigned to the same cluster have the same expected value, which deviates from the overall expected value for the distribution. As each deviation is induced by a common factor, the deviations are sampled from a common distribution.
The assumed model is (hence) an independent mixed effects model (intercept and random effects) for the outcomes under each of the $A$ distributions. 

For $d\in\mathbb{N},$ let $\bC_{a,t}\in\{0,1\}^d$ be a vector denoting cluster assignment for outcome $O_t^a$ such that $\sum_{i=1}^d~C_{a,t,i}~=~1$. 
We assume for all $t$ that
\begin{equation}
O_t^a =  \theta_{a} + \bC_{a,t}^\top\bu_{a} + \epsilon_{a,t},\label{Eqn:Ymodel}
\end{equation}
where, independently,
$$\epsilon_{a,t}\sim \mathcal{N}(0,v^{(1)}_{a})\text{ and } u_{a,i}\sim \mathcal{N}( 0, v^{(2)}_{a}).$$
The set of model parameters for each arm $a$ hence consists of $(\theta_a, \bu_a,v^{(1)}_{a},v^{(2)}_{a})$, and no parameters are shared between the arms.
The process of cluster assignment $(\bC_{a,t})_t$ is assumed predictable, i.e., all cluster assignments are known prior to assignment to the arm. 
The prior specification is as follows, we assume a normal $\mathcal{N}(\theta_0,\sigma_0^2)$ prior on each $\theta_a$, an inverse-gamma $\mathcal{IG}(\alpha_0,\beta_0)$ prior on each $v^{(1)}_{a}$,  and an $\mathcal{IG}(\alpha_1,\beta_1)$ prior on each $v^{(2)}_{a}$.

The above data model and prior specification lead to an analytic expression for the full conditional distribution of each parameter, and hence an efficient Gibbs sampling procedure such as the one in~\citet{wang1993marginal} can be constructed. 
This Gibbs sampling procedure can then be included in a sequential Markov chain Monte Carlo method (\citet{chopin2002sequential}) in order to efficiently update approximations of the posterior distribution $\Pi_t^a$ upon sampling a new observation $O_{t+1}^a$.
For the sequential Markov chain Monte Carlo method  the set of observations $(O^a_t)_t$ leads to a collection of samples $(\theta_{a,i,t}, \bu_{a,i,t}, v^{(1)}_{a,i,t}, v^{(2)}_{a,i,t})_{i=1}^{d_2}$ of particles and weights $(w_{a,i,t})_{i=1}^{d_2}$ such that $\sum_{i=1}^{d_2}w_{a,i,t}=1$~and, denoting with $\delta_x$ the Dirac measure at $x$, \begin{equation}
\hat{\Pi}_t^a=\sum_{i=1}^{d_2} w_{a,i,t}\; \delta_{(\theta_{a,i,t},\bu_{a,i,t}, v^{(1)}_{a,i,t}, v^{(2)}_{a,i,t} )}\approx \Pi_t^a.\label{eqn:approx_posterior}
\end{equation}

Based on this approximation to the posterior, we consider three policies for the Bayesian \mbox{multi-armed} bandit:
\begin{itemize}
\item {\bf SBGIAP}: 
Determine the SBGIA by sampling future approximations $\hat{\Pi}^a_t$ to the posterior ${\Pi}^a_t$ from the Markov chain  that approximates \eqref{transkern:posterior} with transition kernel
$$\mathbb{P}_{\hat{\Pi}^a_t}(E) = \int_{\mathcal{O}} \mathbb{Q}(\hat{\Pi}^a_{t+1}(\cdot\mid O,\hat{\Pi}_t^a)\in E) p(O\mid\hat{\Pi}_t^a)\mu(dO) \;\;\;\; \forall E\in\mathcal{M},$$
where $\mathbb{Q}$ denotes the measure on approximate posteriors induced by a sequential Monte Carlo step using $d_2$ particles, given the
current  approximation to the posterior $\hat{\Pi}_t^a$  and the sampled outcome $O$.
The reward function for the Markov chain is given by the posterior mean under the empirical distribution $$R(\hat{\Pi}_t^a) = \sum_{i=1}^{d_2} w_{a,i,t}( \theta_{a,i,t} + \bC^\top_{a,t}\bu_{a,i,t}).$$
As in Theorem~\ref{Thm:epsilonoptimal}, the SBGIAP now chooses $A_t = \underset{a\in [A]}{\arg\max}\;\nu_{M}(\hat{\Pi}_t^a)$, where $\nu_m$ is determined as the  $M$-th iterate of \eqref{stoch_approx} for a choice of $K,N,n,M$. To decrease the numerical burden, the approximated posteriors $\hat{\Pi}_{t+1}^a,\dots,\hat{\Pi}_{t+N}^a $ in the SBGIAP are based on $d_3\leq d_2$ samples after sampling the first observation $O$, by sampling $d_3$ particles to continue with from the initial distribution~$\hat{\Pi}_t^a$. 
\item {\bf Thompson sampling (\citet{thompson1933likelihood})}: \\
Sample $i\sim \text{Categorical}_{d_2}(\bw_{a,t})$ and set \hbox{$\eta_{a,t} = \theta_{a,i,t} + \bC_{a,t+1}^\top\bu_{a,i,t}.$} Choose $A_t = \underset{a\in [A]}{\arg\max}\;\eta_{a,t}$.
\item {\bf Bayesian upper confidence bound (Bayes-UCB) (\citet{kaufmann2012bayesian})}: 
Set $$\eta_{a,t} = \hat{q}((\theta_{a,i,t} + \bC_{a,t+1}^\top\bu_{a,i,t}, w_{a,i,t})_{i=1}^{d_2}, 1-(t\log(T)^6)^{-1}),$$ where $\hat{q}((\rho_i,w_i)_i,1-\alpha)$ is the empirical $1-\alpha$ quantile given the samples $\rho_i$ and weights $w_i$, and where $T$ is the total sample size of the experiment. Choose $A_t = \underset{a\in [A]}{\arg\max}\;\eta_{a,t}$.
\end{itemize}
The  total discounted rewards found under the policies are compared using a simulation study.
We note that the policies Bayes-UCB and Thompson sampling, unlike the SBGIAP, are not tuned to a specific discount factor~$\gamma$. Other Bayesian bandit policies tuned to a specific discount factor are not known from literature, and introducing them in the current paper would deviate attention from the SBGIAP. Policies Bayes-UCB and Thompson sampling have good performance guarantees for undiscounted reward (\citet{kaufmann2012thompson, kaufmann2012bayesian}), hence in order to have a fair comparison, we compare the performance of the three policies when higher discount factors $\tilde{\gamma} \in \{0.8, 0.9, 0.99\}$ are used to determine the total discounted reward, while we tune the SBGIAP to $\gamma=0.8$.
An outperformance over Bayes-UCB and Thompson sampling in terms of Bayesian total discounted reward (with discount factor $\tilde{\gamma} = \gamma$) is expected for the Gittins index policy, as it exactly maximizes this quantity. 
The SBGIA is however an approximation to the Gittins index, and hence outperformance for the SBGIAP in terms of the Bayesian total discounted reward is not guaranteed, furthermore there is no guarantee that a policy based on the Gittins index tuned to $\gamma=0.8$ also outperforms other policies for other discount factors $\tilde{\gamma}$. Hence, it is interesting to compare the performance of the policy using a simulation study.

For the simulation study, we set $\theta_0=0$, $\sigma^2_0 = 1$, $\alpha_{0} = 13$, $\beta_0 = 12$, $\alpha_{1}= 6$, $\beta_{1}=10$.
In order to approximate the Bayesian total discounted reward, the parameters were sampled from the resulting prior distributions for each simulation. 
The sample size of the simulation study was set to $T=300$, the number of clusters $d$ was set to 3, and the number of arms $A$ was set to $3$.
The cluster assignments $\bC_a$ were sampled uniformly for each arm $a$. Given the sampled parameters and cluster assignments, the vector of observations $\bO^a\in\mathbb{R}^{300}$ were sampled according to model \eqref{Eqn:Ymodel}. Given the data and cluster assignments, a sequence of weights $(\bw_{a,t})_{t=1}^{300}$ ($d_2=100$) and particles $( \btheta_{a,t}, \bsigma^2_{a,t}, \btau_{a,t},\bu_{a,t})_{t=1}^{300}$ was generated  for each arm using sequential Markov chain Monte Carlo sampling (starting with a sample from the prior) using 5 Gibbs sampler iterations in each Markov chain Monte Carlo step. 
Each algorithm then determined an interleaving of these independent Markov chain samples, where for the SBGIAP, we set \hbox{$K=1$}, \hbox{$N=25$}, each \hbox{$n(\bi,k) = 1$}, \hbox{$M=100$}, and \hbox{$d_3 = 3$.} 
The above procedure, resulting in an interleaving of the sampled Markov chain $(\bw_{a,t}, \btheta_{a,t},\bu_{a,t}, \bsigma^2_{a,t}, \btau_{a,t})_{t=1}^{300}$ for each algorithm described above, is then repeated independently 2500 times to approximate the Bayesian total discounted reward. 
This procedure took about 10 days on a computer with 32 cores.
\begin{figure}[h!]
\centering
\includegraphics[width = \textwidth]{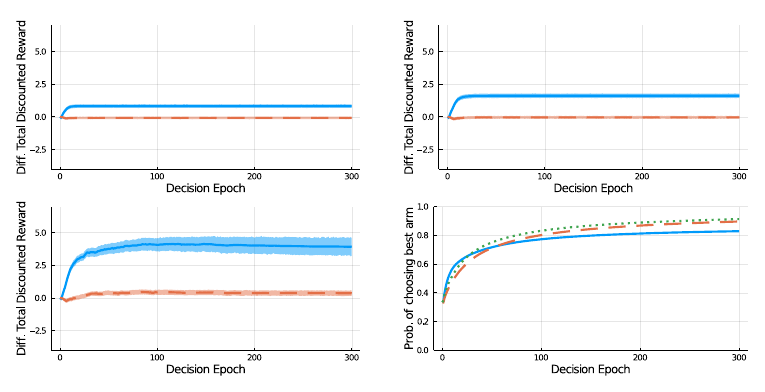}
\caption{Results for  the SBGIAP for the Gaussian random effects Bayesian multi-armed bandit model. The difference in running total discounted reward for a discount factor of $0.8$ (top-left), $0.9$ (top-right), and $0.99$ (bottom-left) is shown for the SBGIAP vs. Bayes-UCB (solid) and Thompson sampling vs. Bayes-UCB (dashed). The running frequency of choosing the best arm for the SBGIAP (solid), Thompson sampling (dashed) and Bayes-UCB (dotted) is shown on the bottom-right.
	\label{comparison_80perc}}

\end{figure}

Figure~\ref{comparison_80perc} shows results for the SBGIAP with $\gamma=0.8$. 
Differences (averaged over 2500 simulations) between Bayesian total discounted reward for the SBGIAP vs. Bayes-UCB  (solid, blue)  and Thompson sampling vs. Bayes-UCB (dotted, red) are shown for discount factors \hbox{$\tilde{\gamma} = 0.8,0.9$, and $0.99$,} along with a point-wise $95\%$ bootstrapped confidence interval for the mean. The total discounted rewards at each decision epoch are calculated as the discounted sum of the expected rewards ($\mathbb{E}[O_t] = \theta_{A_t} + C_{A_t,t}\bu_a$) from the initial decision epoch up to that time point. 

Figure~\ref{comparison_80perc} shows that the SBGIAP significantly outperforms Bayes-UCB and Thompson sampling with a final difference in average total discounted reward of about 1, 2 and 4 for discount factors $0.8,0.9$, and $0.99$ respectively.


The average (undiscounted) frequency of choosing the best arm is calculated as the frequency each algorithm chose $a_t^*=\underset{a}{\arg\max}\;\theta_a + \bC_{at}^\top\bu_a$ at each decision epoch.
It is seen that while the SBGIAP reaches a frequency of $60\%$ of optimal pulls early on, the other strategies end up at a higher frequency of optimal pulls. 
This might be a result of the low discount factor used for constructing the policy, indicating that short-term gains are preferred over long-term ones. The bottom-right graph is in line with the other graphs in Figure~\ref{comparison_80perc} as, when considering the sum of discounted rewards starting at the initial state, making good choices early on has a larger benefit than outperformance in the long run.

We conclude that with the SBGIAP, significant outperformance in terms of total discounted reward with respect to~state-of-the-art policies can be attained for more complex models than usually considered in bandit literature.

\section{Discussion and conclusion}
In this paper, we have proposed the sampling-based Gittins index approximation (SBGIA).
In Section \ref{sect: theoretical_results}, the SBGIA was introduced, and a general error bound was shown for the Gittins index obtained when truncating the horizon and support for the rewards using a stopping time, which can be viewed as an extension of the results presented in~\citet{wang1997error}. Next, finite and asymptotic convergence results were shown for the SBGIA. 
Using the finite-time convergence result it is possible to obtain a confidence interval for the Gittins index which holds under a finite number of stochastic approximation samples.
Next, it was shown in Theorem~\ref{Thm:epsilonoptimal} that by making explicit choices for the width and safety level $\delta_2$ of this confidence interval, the policy choosing the largest Gittins index estimate is an $\epsilon$-optimal policy for the Bayesian multi-armed bandit problem. 
For both the Bernoulli and Gaussian bandits, the SBGIA was seen to yield a good approximation to the Gittins index, increasing in quality with the number of nested simulations and the truncation parameter $K$, and showing the best approximation when estimating the limit as $K$ goes to infinity. The results indicated that an efficient strategy for Gittins index approximation using SBGIA is to first increase the number of nested simulations, and then increase the truncation parameter $K$, until no significant differences in the estimate are seen for both steps. 
The SBGIA can be applied even in cases where the actual transition kernel is unknown, but where samples from an approximation to the transition kernel can be generated.
An example of this was seen in Section \ref{Subsec:Randomeffects}, where samples from the approximate posterior were generated using sequential Monte Carlo sampling. 
In this case, the SBGIAP was seen to outperform the state-of-the-art policies Bayes-UCB and Thompson sampling in terms of Bayesian total discounted reward. 

The SBGIA can be applied to any family of alternative bandit processes. For example, to compute the Gittins index approximations in Section~\ref{sec:52}, only three things must be altered for each bandit, namely the transition kernel, the reward function, and calculation of the bounds $R_\ell,R_u.$ 
In contrast, for the Calibration method, an additional requirement is that the reward support also has to be discretized and the change in state space leads to a reformulation of the backward induction step. The benefit of a method that works in general, is that there is more flexibility in the model choice when basing treatment allocation on the Gittins index, as there is no increased difficulty in implementing the calculation method when assuming a more elaborate model for the data. Another benefit is that less expert knowledge is necessary for Gittins index approximation. It might be an interesting idea to have a software library where practitioners only have to input functions that calculate, e.g., the posterior mean, after which the package calculates the SBGIA. It is furthermore useful to have a method that does not assume known transition probabilities, as in many real-life cases the posterior distribution cannot be calculated in closed form because of the high dimensionality of the model or when the assumed prior is nonconjugate.

In Section \ref{Subsec:Randomeffects}, we evaluated the performance of the SBGIA policy (SBGIAP) for a novel random effects bandit problem. The SBGIAP was defined based on an approximation of the Markov chain~$(\Pi_t^a)_t$ describing the evolution of the posterior distribution, based on sequential Markov chain Monte Carlo. In future research, it would be interesting to investigate how finite-time convergence results for Markov chains (e.g., \citet{rosenthal1995convergence}) can be used to construct finite-time error bounds for the SBGIA in these situations. 
In this paper, the SBGIA was evaluated for a number of Bayesian multi-armed bandit problems.
In the case of Gaussian outcomes, the Gittins index was shown to result in near-optimal frequentist undiscounted regret (\citet{lattimore2016regret}). If this result is shown to hold in general, the confidence interval presented in Theorem~\ref{FTCI} ensures that we have a method that can approximate, up to arbitrary precision, a near-optimal policy, in terms of undiscounted frequentist regret, for the multi-armed bandit problem.

%
%
%
\bibliographystyle{plainnat} 
\bibliography{Baas_Boucherie_Braaksma_SBGIA.bib} 

\appendix

\section{Proofs of theorems} \label{proofs}

\setcounter{theorem}{2}
\begin{theorem} \thmBoundtrunc{defGI_restricted2}
\end{theorem}
\begin{proof}
Let $(\tilde{\bS}_t)_t$ be a Markov chain with states $\tilde{\bS}_t = (\bS_u)_{u=0}^t$. Let $\tilde{R}(\tilde{\bS}_t) = R(\tilde{\bS}_{t})\mathbb{I}(t< \sigma )$. The pair $(
\tilde{\bS},\tilde{R})$ defines a Markov reward process as $\mathbb{I}(t<\sigma )$ is a function of $\tilde{\bS}_t$.
As the filtrations and hence the set of stopping times generated by $\bS$ and $\tilde{\bS}$ are the same, we have \begin{equation}
	\sup_{\tau\in\mathcal{T}}\mathbb{E}_{\bs}\left[\sum_{t=0}^{\tau-1}\gamma^t (\tilde{R}(\tilde{\bS}_t)-\nu)\right]=\sup_{\tau\in\mathcal{T}}\mathbb{E}_{\bs}\left[\sum_{t=0}^{\tau\wedge \sigma-1}\gamma^t (R(\bS_t)-\nu)\right] = \sup_{\tau\in\mathcal{T}_\sigma}\mathbb{E}_{\bs}\left[\sum_{t=0}^{\tau-1}\gamma^t (R(\bS_t)-\nu)\right].\label{eqproofthm1}
\end{equation}
Note that as Assumption \ref{boundedtotrew} holds for the Markov reward process defined by $(\bS,R)$, it also holds for the Markov reward process  defined by $(
\tilde{\bS},\tilde{R})$. Hence from  \citet{lattimore2020bandit} there is an unique value for $\nu$, denoted by $\nu_\sigma(\bs)$, such that the left-hand side, hence the right-hand side of \eqref{eqproofthm1} is zero. As $f_\bs^\sigma$ is a scaling of \eqref{eqproofthm1}, the first two statements of the theorem are proven.  
From \citet{lattimore2020bandit}, we also have that the supremum in \eqref{eqproofthm1} is attained for an unique stopping time $\tau(\bs,\nu)$, proving the third statement. 

Now, we bound the approximation error from above (the lower bound holds trivially):
\begin{align*}
	\nu(\bs) - \nu_\sigma(\bs) &= \sup_{\tau\in\mathcal{T}}\frac{\mathbb{E}_\bs[\sum_{t=0}^{\tau-1}\gamma^tR(\bS_t)]}{\mathbb{E}_\bs[\sum_{t=0}^{\tau-1}\gamma^t]}-\sup_{\tau\in\mathcal{T}}\frac{\mathbb{E}_\bs[\sum_{t=0}^{\tau\wedge \sigma-1}\gamma^tR(\bS_t)]}{\mathbb{E}_\bs[\sum_{t=0}^{\tau-1}\gamma^t]}\\&\leq 
	\sup_{\tau\in\mathcal{T}}\frac{\mathbb{E}_\bs[\sum_{t=0}^{\tau-1}\gamma^tR(\bS_t)] - \mathbb{E}_\bs[\sum_{t=0}^{\tau\wedge \sigma-1}\gamma^tR(\bS_t)]}{\mathbb{E}_\bs[\sum_{t=0}^{\tau-1}\gamma^t]}\\&= \sup_{\tau\in\mathcal{T}}\frac{\mathbb{E}_\bs[\sum_{t=\sigma\wedge \tau}^{\tau-1}\gamma^tR(\bS_t)]}{\mathbb{E}_\bs[\sum_{t=0}^{\tau-1}\gamma^t]}=\sup_{\substack{\tau\in\mathcal{T}\\
			\tau\geq \sigma}}\frac{\mathbb{E}_\bs[\sum_{t=\sigma}^{\tau-1}\gamma^tR(\bS_t)]}{\mathbb{E}_\bs[\sum_{t=0}^{\tau-1}\gamma^t]}\\
	& \leq \sup_{\substack{\tau\in\mathcal{T}\\
			\tau\geq \sigma}}\frac{\mathbb{E}_\bs[\sum_{t=\sigma}^{\tau-1}\gamma^tR(\bS_t)]}{\mathbb{E}_\bs[\sum_{t=0}^{\sigma-1}\gamma^t]} = \sup_{\substack{\tau\in\mathcal{T}\\
			\tau\geq \sigma}}\frac{(1-\gamma)\mathbb{E}_\bs[\gamma^\sigma \mathbb{E}_\bs[\sum_{t=\sigma}^{\tau-1}\gamma^{t-\sigma}R(\bS_t)|\mathcal{F}_\sigma]]}{1 -\mathbb{E}_\bs[\gamma^\sigma]}\\
	&\leq \sup_{\substack{\tau\in\mathcal{T}\\
			\tau\geq \sigma}}\frac{\mathbb{E}_\bs\left[\gamma^\sigma \left(\frac{\mathbb{E}_{\bS_\sigma}[\sum_{t=0}^{\tau-\sigma}\gamma^{t}R(\bS_t)]}{\mathbb{E}_{\bS_\sigma}[\sum_{t=0}^{\tau-\sigma} \gamma^t]}\right)^+\right]}{1 -\mathbb{E}_\bs[\gamma^\sigma]}\leq \frac{\mathbb{E}_\bs\left[\gamma^\sigma \nu(\bS_\sigma)^+\right]}{1 -\mathbb{E}_\bs[\gamma^\sigma]}.
\end{align*}
The second to last inequality above holds by the strong Markov property and as $(1-\gamma)\leq (1-\gamma)/(1-\mathbb{E}_{\bS_\sigma}[\gamma^{\tau-\sigma+1}])$ almost surely. The ``positive part" of the Gittins index in the last term above comes from the fact that we can take $\tau = \sigma$ in the seventh term.
\end{proof}
\setcounter{theorem}{3}
\begin{theorem}
\thmFTCB
\end{theorem}
\begin{proof}
Let $\mathcal{F}_m^\epsilon = \sigma(\epsilon_0,\dots, \epsilon_{m})$ such that $\omega_m$ and $\alpha_m$ are $\mathcal{F}_{m-1}^\epsilon$-measurable. Then, as $\mathbb{E}[\epsilon_{m}|\mathcal{F}^\epsilon_{m-1}]~=~0$ for $m\geq 1$, we have
\begin{align}
	\mathbb{E}[(\omega^*-\omega_{m+1})^2|\mathcal{F}^\epsilon_{m-1}] &=   \mathbb{E}[(\omega^*-\omega_{m} + \alpha_m(f_\bs(\omega_{m})-f_\bs(\omega^*)) + \alpha_m\epsilon_m)^2|\mathcal{F}^\epsilon_{m-1}] \nonumber\\&=
	(\omega^*-\omega_{m} + \alpha_m(f_\bs(\omega_{m})-f_\bs(\omega^*))^2 + \alpha_m^2\mathbb{E}[\epsilon_m^2].\label{partialderiv}
\end{align}
Now, by \eqref{Ineq: bounds_f}, letting $\eta(\omega_1,\omega_2) = 1+\left(\frac{1-\gamma}{1-\gamma^N} -1\right)\mathbbm{1}_{[\omega_2\geq \omega_1]}$ we have
\begin{align*}
	&\left(1-\frac{\alpha_{m}\eta(\omega^*,\omega_m)}{2(R_u-R_\ell)}\right)(\omega^*-\omega_m)\leq\omega^*-\omega_m +\alpha_{m}(f_\bs(\omega_m)-f_\bs(\omega^*))\leq \left(1-\frac{\alpha_{m}\eta(\omega_m,\omega^*)}{2(R_u-R_\ell)}\right)(\omega^*-\omega_m)\end{align*}
Hence \begin{align*}
	&|\omega^*-\omega_m +\alpha_{m}(f_\bs(\omega_m)-f_\bs(\omega^*))|\\&\leq \max\left(\left(1-\frac{\alpha_{m}\eta(\omega^*,\omega_m)}{2(R_u-R_\ell)}\right)(\omega_m-\omega^*),\left(1-\frac{\alpha_{m}\eta(\omega_m,\omega^*)}{2(R_u-R_\ell)}\right)(\omega^*-\omega_m) \right)\\&=\left(1-\frac{\alpha_{m}(1-\gamma)}{2(R_u-R_\ell)(1-\gamma^N)}\right)|\omega^*-\omega_m| .
\end{align*}
The last line above follows as $\eta\in[0,1]$, hence $\left(1-\frac{\alpha_{m}\eta(x,y)}{2(R_u-R_\ell)}\right)\geq 0 $ for all $x,y$ by the assumptions on $(\alpha_m)_m$ and the maximum will be attained at the first argument if $\omega_m\geq \omega^*$ where $\eta(\omega^*,\omega_m) = \frac{1-\gamma}{1-\gamma^N}$ and the maximum will be attained at the second argument if $\omega^*\geq \omega_m$ where $\eta(\omega_m,\omega^*) = \frac{1-\gamma}{1-\gamma^N}$.
Then, continuing \eqref{partialderiv}, we have:
\begin{equation*}
	\mathbb{E}[(\omega^*-\omega_{m+1})^2|\mathcal{F}^\epsilon_{m-1}] \leq \left(1-\frac{\alpha_{m}(1-\gamma)}{2(R_u-R_\ell)(1-\gamma^N)}\right)^2(\omega^*-\omega_m)^2 +\alpha_m^2\mathbb{E}[\epsilon_m^2].
\end{equation*}
To conclude, taking expectations, we have
\begin{equation*}
	\mathbb{E}[(\omega^*-\omega_{m+1})^2] \leq \mathbb{E}[\left(1-c\alpha_m\right)^2(\omega^*-\omega_m)^2] +\mathbb{E}[\alpha_m^2]\mathbb{E}[\epsilon_m^2].
\end{equation*}
where $c = \frac{1-\gamma}{2(R_u-R_\ell)(1-\gamma^N)}$.
\end{proof}

\setcounter{lemma}{0}
\begin{lemma}\label{lemma_borkar1cond}
\Lemmanubound
\end{lemma}

\begin{proof}
Note that if  $\nu> R_u$ it holds that irrespective of the path $\bS$  $$\min_{u\in [N]}Z_{\bi,u,n}^{(1)}(\nu)  =\min_{u\in [N]}\;c \sum_{w=0}^{(u\wedge \sigma_H)-1}\gamma^w (\nu-R(\bS_w)) = c  (\nu-R(\bS_0))=Z_{\bi,1,n}^{(1)}(\nu)>0.$$ For $k=2$, we hence have irrespective of the path $\bS$, that when $\nu> R_u$ 
\begin{align*}\min_{u\in [N]}Z_{\bi,u,n}^{(2)}(\nu) &= \min_{u\in [N]}Z_{[\bi,1],u,n}^{(1)}(\nu) - \frac{1}{n(\bi,2)}\sum_{j=2}^{n(\bi,2)+1}\left(\min_{w\in [N]}Z_{[\bi,j],w,n}^{(1)}(\nu)\Big|\{\bS^{[\bi,j]}_{\ixv{u}}=\bS^\bi_{\ixv{u}}\}\right)\\& = \min_{u\in [N]}Z_{[\bi,1],u,n}^{(1)}(\nu) - \left(Z_{[\bi,j],1,n}^{(1)}(\nu)\Big|\{\bS^{[\bi,j]}_{\ixv{u}}=\bS^\bi_{\ixv{u}}\}\right)=0. \end{align*}

From this, we see that for $k> 2$ $$\min_{u\in [N]}Z_{\bi,u,n}^{(k)}(\nu) = \min_{u\in [N]}Z_{[\bi,1],u,n}^{(k-1)}(\nu) - \frac{1}{n(\bi,k)}\sum_{j=2}^{n(\bi,k)+1}\left(\min_{w\in [N]}Z_{[\bi,j],w,n}^{(k-1)}(\nu)\Big|\{\bS^{[\bi,j]}_{\ixv{u}}=\bS^\bi_{\ixv{u}}\}\right)=0 - 0 =0.$$

Hence for $\nu> R_u$ we have almost surely \begin{align*}
	V^{(K)}_{\bs}(\nu) &= \max\left(-1/2,\;\min\left(1/2,\; \frac{1}{n(K,K)}\sum_{j=1}^{n(K,K)}\sum_{k=1}^K \min_{u\in [N]}Z_{[k,j],u, n}^{(k)}\right)\right)\\
	& = \max\left(-1/2,\;\min\left(1/2,\; \frac{1}{n(K,K)}\sum_{j=1}^{n(K,K)} Z_{[1,j],1, n}^{(1)}\right)\right)
	>0.
\end{align*}  
Similarly, observe that $V^{(K)}_{\bs}(\nu)<0$ almost surely if $\nu< R_\ell$. Hence by \eqref{stoch_approx} we have almost surely for all $m$ that
$$\nu_m\in[R_\ell-M_\alpha/2,\; R_u +M_\alpha/2].$$

\end{proof}

\setcounter{theorem}{8}
\setcounter{lemma}{3}
\begin{lemma} \label{thm:deriv}\Thmderiv


\end{lemma}
\begin{proof}

First, we show by induction that for all $k$, $\nu$, random times $T$ and sets $E$ in $\times_\bi\sigma(\bS^{\bi})$ (the smallest product sigma algebra measuring all versions $\bS^\bi$) we have
\begin{align*}
	&Z_{\bi,T,n}^{(k)}(\nu)\mid E =   A^{(k)}_{\bi,T,n}(\nu)\mid E+\nu \cdot B^{(k)}_{\bi,T,n}(\nu)\mid E
\end{align*}
where $A^{(k)}_{\bi,t,n}$, $B^{(k)}_{\bi,t,n}$ are random variables  only depending on $\nu$ through the minimizers $U^{(k')}_{\bi,n}(\nu)$ for $k'\leq k$.

For $k=1$ the result is immediately verified for $A_{\bi,t,n}^{(1)}(\nu)=-\frac{(1-\gamma)}{(R_u-R_\ell)(1-\gamma^N)}\sum_{u=0}^{t-1}\gamma^u R(\bS_u^\bi)$ and $B^{(1)}_{\bi,t,n}(\nu)=b_t$.

Now assume the result holds up to $k$, then by \eqref{defZk}, as $\{\bS^{[\bi,j]}_{\ixT}= \bS^{[\bi]}_{\ixT}\}\in\times_\bi \sigma(\bS^\bi)$ for all $j$
\begin{align*}
	Z_{\bi,T,n}^{(k)}(\nu)\mid E
	= A^{(k)}_{\bi,T,n}\mid E+\nu\cdot B^{(k)}_{\bi,T,n}\mid E
\end{align*}
for  \begin{align*}&A^{(k)}_{\bi,T,n} = A^{(k-1)}_{[\bi,1],T,n} -\frac{1}{n(\bi,k)}\sum_{j=2}^{n(\bi,k)+1}A^{(k-1)}_{[\bi,j],U_{[\bi,j],n}^{(k-1)}(\nu),n}|\{\bS^{[\bi,j]}_{\ixT}= \bS^{[\bi]}_{\ixT}\},\\& B^{(k)}_{\bi,T,n}
	=B^{(k-1)}_{[\bi,1],T,n}-\frac{1}{n(\bi,k)}\sum_{j=2}^{n(\bi,k)+1}B^{(k-1)}_{[\bi,j],U_{[\bi,j],n}^{(k-1)}(\nu),n}|\{\bS^{[\bi,j]}_{\ixT}= \bS^{[\bi]}_{\ixT}\}.\end{align*}

Second, we show that all but finitely many points $\nu$ we have $$ \mathbb{P}\left(\lim_{\nu'\rightarrow\nu}U^{(k)}_{\bi,n}(\nu')=U^{(k)}_{\bi,n}(\nu)\right)=1.$$ 

For $k\geq 2$ and $\bt \in [N]^{\prod_{k'=1}^{k-1}(1+n(\bi,k'))}$ let $A^{(k)}_{\bi,\bt,n}$ be $A^{(k)}_{\bi,T,n}$ where every nested random time is replaced by the elements of $\bt$ consecutively, and let $b^{(k)}_{\bi,\bt,n}$ be defined similarly for $B^{(k)}_{\bi,T, n}$, note that this makes $b^{(k)}_{\bi,\bt,n}$ deterministic.

For $\nu_2<\nu_1$
\begin{align*}
	&\mathbb{I}(U_{\bi,n}^{(k)}(\nu_1) \neq U^{(k)}_{\bi,n}(\nu_2)) \leq  \mathbb{I}(\exists t,u\;\; Z_{\bi,t,n}^{(k)}(\nu_1) \leq Z_{\bi,u,n}^{(k)}(\nu_1) ,\; Z_{\bi,t,n}^{(k)}(\nu_2) > Z_{\bi,u,n}^{(k)}(\nu_2) )\\&\leq
	\mathbb{I}\Bigg( \exists\; \bt,\, \bu,\, A^{(k)}_{\bi,\bt,n}+\nu_1\cdot b^{(k)}_{\bi,\bt,n}\leq A^{(k)}_{\bi,\bu,n}+\nu_1\cdot b^{(k)}_{\bi, \bu,n},\;
	A^{(k)}_{\bi,\bt,n}+\nu_2\cdot b^{(k)}_{\bi,\bt,n}> A^{(k)}_{\bi,\bu,n}+\nu_2\cdot b^{(k)}_{\bi, \bu,n}\Bigg)\\&\leq \sum_{\bt,\bu}\mathbb{I}\Bigg( A^{(k)}_{\bi,\bt,n}+\nu_1\cdot b^{(k)}_{\bi,\bt,n}\leq A^{(k)}_{\bi,\bu,n}+\nu_1\cdot b^{(k)}_{\bi, \bu,n},\;
	A^{(k)}_{\bi,\bt,n}+\nu_2\cdot b^{(k)}_{\bi,\bt,n}> A^{(k)}_{\bi,\bu,n}+\nu_2\cdot b^{(k)}_{\bi, \bu,n}\Bigg)\\& = \sum_{\bt,\bu}\mathbb{I}\Bigg(\frac{A^{(k)}_{\bi,\bt,n}-A^{(k)}_{\bi,\bu,n}}{b^{(k)}_{\bi,\bu,n}-b^{(k)}_{\bi,\bt,n}}\in(\nu_2,\nu_1]\Bigg)
\end{align*}
The indicators above converge to $\mathbb{I}\Bigg(\frac{A^{(k)}_{\bi,\bu,n}-A^{(k)}_{\bi,\bt,n}}{b^{(k)}_{\bi,\bt,n}-b^{(k)}_{\bi,\bu,n}}=\nu_1\Bigg)$ when $\nu_2\uparrow\nu_1$ which has nonzero probability of being equal to one for at most finitely many points $\nu_1$ by the assumption. The direction $\nu_1\downarrow \nu_2$ can be shown by changing the first bound above to $$\mathbb{I}(\exists t,u\;\; Z_{\bi,t,n}^{(k)}(\nu_1) < Z_{\bi,u,n}^{(k)}(\nu_1) ,\; Z_{\bi,t,n}^{(k)}(\nu_2) \geq Z_{\bi,u,n}^{(k)}(\nu_2) )$$ and following the same steps as above.

Third, we show for points $\nu_1$ where  all functions $U^{(k')}_{\bi,n}$ are continuous for $k'\leq k$ that
\begin{equation}
	\frac{d}{d\nu} \min_{u\in[N]}Z_{[k,1],u,n}^{(k)}(\nu) = h_{[k,1],n}^{(k)}(U^{(k)}_{[k,1],n}(\nu),\nu).\label{result_deriv}
\end{equation} 
Note that, by the second step above, the complement of this set is a subset of the set $\mathcal{V}$ of jump points for the cumulative distribution functions of $(A^{(k')}_{\bi,\bu,n}-A^{(k')}_{\bi,\bt,n})(b^{(k')}_{\bi,\bt,n}-b^{(k')}_{\bi,\bu,n})$ for all $\bi, \bt,\bu,k'\leq k.$

We show the result by induction. Let $k=1$, then we have for $\nu_1> \nu_2$
\begin{align*}
	&\frac{\min_{u\in[N]}Z_{[1,1],u,n}^{(1)}(\nu_1) - \min_{u\in[N]}Z_{[1,1],u,n}^{(1)}(\nu_2) }{\nu_1-\nu_2} \geq \frac{Z_{[1,1],U_{[1,1],n}^{(1)}(\nu_1),n}(\nu_1) - Z_{[1,1],U_{[1,1],n}^{(1)}(\nu_1),n}(\nu_2) }{\nu_1-\nu_2} = b_{U_{[1,1],n}^{(1)}(\nu_1)},  \\
	&\frac{\min_{u\in[N]}Z_{[1,1],u,n}^{(1)}(\nu_1) - \min_{u\in[N]}Z_{[1,1],u,n}^{(1)}(\nu_2) }{\nu_1-\nu_2} \leq \frac{Z_{[1,1],U_{[1,1],n}^{(1)}(\nu_2),n}(\nu_1) - Z_{[1,1],U_{[1,1],n}^{(1)}(\nu_2),n}(\nu_2) }{\nu_1-\nu_2} =b_{U_{[1,1],n}^{(1)}(\nu_2)}.
\end{align*}
By the choice of $\nu_1$ we have that the upper and lower bound almost everywhere almost surely converge to $h_{[1,1],n}^{(1)}(U^{(1)}_{[1,1],n}(\nu_1),\nu_1) = b_{U_{[1,1],n}^{(1)}(\nu_1)}$ when $\nu_2\uparrow\nu_1$. The case $\nu_1<\nu_2$ can be shown similarly.

Let the above statement hold up to $k$, then for $\nu_1>\nu_2$
\begin{align*}
	&\min_{u\in[N]}Z^{(k)}_{[k,1], u,n}(\nu_1)-\min_{u\in[N]}Z^{(k)}_{[k,1], u,n}(\nu_2)\\
	&=Z^{(k-1)}_{[k,1,1], U_{[k,1],n}^{(k)}(\nu_1),n}(\nu_1) - \frac{1}{n(\bi,k)}\sum_{j=2}^{n(\bi,k)+1}\left(Z^{(k-1)}_{[k,1,j],U_{[k,1,j],n}^{(k-1)}(\nu_1),n}\left(\nu_1
	\right)\Big|\{\bS^{[k,1,j]}_{\ixv{U_{[k,1],n}^{(k)}(\nu_1)}}= \bS^{[k,1]}_{\ixv{U_{[k,1],n}^{(k)}(\nu_1)}}\}\right)\\&\phantom{=\;} - Z^{(k-1)}_{[k,1,1], U_{[k,1],n}^{(k)}(\nu_2),n}(\nu_2) + \frac{1}{n(\bi,k)}\sum_{j=2}^{n(\bi,k)+1}\left(Z^{(k-1)}_{[k,1,j],U_{[k,1,j],n}^{(k-1)}(\nu_2),n}\left(\nu_2
	\right)\Big|\{\bS^{[k,1,j]}_{\ixv{U_{[k,1],n}^{(k)}(\nu_2)}}= \bS^{[k,1]}_{\ixv{U_{[k,1],n}^{(k)}(\nu_2)}}\}\right)\\&\geq Z^{(k-1)}_{[k,1,1], U_{[k,1],n}^{(k)}(\nu_1),n}(\nu_1) - \frac{1}{n(\bi,k)}\sum_{j=2}^{n(\bi,k)+1}\left(Z^{(k-1)}_{[k,1,j],U_{[k,1,j],n}^{(k-1)}(\nu_2),n}\left(\nu_1
	\right)\Big|\{\bS^{[k,1,j]}_{\ixv{U_{[k,1],n}^{(k)}(\nu_1)}}= \bS^{[k,1]}_{[U_{[k,1],n}^{(K)}(\nu_1)]}\}\right)\\&\phantom{=} - Z^{(k-1)}_{[k,1,1], U_{[k,1],n}^{(k)}(\nu_1),n}(\nu_2) + \frac{1}{n(\bi,k)}\sum_{j=2}^{n(\bi,k)+1}\left(Z^{(k-1)}_{[k,1,j],U_{[k,1,j],n}^{(k-1)}(\nu_2),n}\left(\nu_2
	\right)\Big|\{\bS^{[k,1,j]}_{\ixv{U_{[k,1],n}^{(k)}(\nu_1)}}= \bS^{[k,1]}_{\ixv{U_{[k,1],n}^{(k)}(\nu_1)}}\}\right)\\&= \left(Z^{(k-1)}_{[k,1,1], U_{[k,1],n}^{(k)}(\nu_1),n}(\nu_1) -Z^{(k-1)}_{[k,1,1], U_{[k,1],n}^{(k)}(\nu_1),n}(\nu_2)\right) \\& \phantom{=}  + \frac{1}{n(\bi,k)}\sum_{j=2}^{n(\bi,k)+1}\left(Z^{(k-1)}_{[k,1,j],U_{[k,1,j],n}^{(k-1)}(\nu_2),n}(\nu_1
	)-Z^{(k-1)}_{[k,1,j],U_{[k,1,j],n}^{(k-1)}(\nu_2),n}\left(\nu_2
	\right)\right)\Big|\{\bS^{[k,1,j]}_{\ixv{U_{[k,1],n}^{(k)}(\nu_1)}}= \bS^{[k,1]}_{\ixv{U_{[k,1],n}^{(k)}(\nu_1)}}\}.
\end{align*}
Taking $\nu_2\uparrow\nu_1$ above we see that 
$${\lim\inf}_{\nu_2\uparrow\nu_1}\frac{\min_{u\in[N]}Z^{(k)}_{[k,1], u,n}(\nu_1)-\min_{u\in[N]}Z^{(k)}_{[k,1], u,n}(\nu_2)}{\nu_1-\nu_2}\geq h^{(k)}_{[k,1],n}(U^{(k)}_{[k,1],n}(\nu_1),\nu_1).$$
Similarly, it can be shown that 
\begin{align*}
	&\min_{u\in[N]}Z^{(k)}_{[k,1], u,n}(\nu_1)-\min_{u\in[N]}Z^{(k)}_{[k,1], u,n}(\nu_2)\\&\leq \left(Z^{(k-1)}_{[k,1,1], U_{[k,1],n}^{(k)}(\nu_2),n}(\nu_1) -Z^{(k-1)}_{[k,1,1], U_{[k,1],n}^{(k)}(\nu_2),n}(\nu_2)\right) \\& \phantom{=}  + \frac{1}{n(\bi,k)}\sum_{j=2}^{n(\bi,k)+1}\left(Z^{(k-1)}_{[k,1,j],U_{[k,1,j]}^{(k-1),n}(\nu_1),n}(\nu_1
	)-Z^{(k-1)}_{[k,1,j],U_{[k,1,j],n}^{(k-1)}(\nu_1),n}\left(\nu_2
	\right)\right)\Big|\{\bS^{[k,1,j]}_{\ixv{U_{[k,1],n}^{(k)}(\nu_2)}}= \bS^{[k,1]}_{\ixv{U_{[k,1],n}^{(k)}(\nu_2)}}\}.
\end{align*}
Now, by the choice of $\nu_1$, we know there is a random variable $\epsilon(\nu_1)$ such that $U_{[k,1],n}^{(k)}(\nu_2) = U_{[k,1],n}^{(k)}(\nu_1)$ almost surely for all $\nu_2$ such that $|\nu_1-\nu_2|\leq \epsilon(\nu_1)$. Hence, for every sample path, for $\nu_2$ close enough to $\nu_1$, we condition on $\{\bS^{[k,1,j]}_{\ixv{U_{[k,1],n}^{(k)}(\nu_1)}}= \bS^{[k,1]}_{\ixv{U_{[k,1],n}^{(k)}(\nu_1)}}\}$ above. From this it follows that almost surely
$${\lim\sup}_{\nu_2\uparrow\nu_1}\frac{\min_{u\in[N]}Z^{(k)}_{[k,1], u,n}(\nu_1)-\min_{u\in[N]}Z^{(k)}_{[k,1], u,n}(\nu_2)}{\nu_1-\nu_2}\leq h^{(k)}_{[k,1],n}(U_{[k,1],n}^{(k)}(\nu_1),\nu_1)$$
hence \eqref{result_deriv} holds. The case $\nu_1<\nu_2$ works similarly.

By induction it can be verified that the right-hand side of \eqref{result_deriv} is bounded for each $k$. 

We have for $\nu_1\in\mathcal{V}$ and $\nu_1>\nu_2$
\begin{align*}
	\frac{\tilde{f}_\bs(\nu_1) - \tilde{f}_\bs(\nu_2)}{\nu_1-\nu_2} &= \mathbb{E}\left[\sum_{k=1}^K \frac{\min_{u\in [N]}Z_{[1,j],u, n}^{(k)}(\nu_1) - \min_{u\in [N]}Z_{[1,j],u, n}^{(k)}(\nu_2)}{\nu_1-\nu_2}\mathbbm{I}((\tilde{V}_\bs(\nu_1),\tilde{V}_\bs(\nu_2))\in [-1/2,1/2]^2)\right]\\&+
	\mathbb{E}\left[ \frac{V_\bs(\nu_1) - V_\bs(\nu_2)}{\nu_1-\nu_2}\mathbbm{I}((\tilde{V}_\bs(\nu_1),\tilde{V}_\bs(\nu_2))\notin [-1/2,1/2]^2)\right]
	. 
\end{align*}
When $\nu_2\uparrow\nu_1$ by continuity and boundedness of the derivative the first term goes to 
$$\sum_{k=1}^K\mathbb{E}_\bs\left[h^{(k)}_{[k,1],n}\left(U^{(k)}_{[k,1],n}(\nu),\nu\right)\mathbbm{I}(\tilde{V}_{\bs}(\nu)\in [-1/2,1/2])\right]$$
while the second term goes to zero by continuity (of $\tilde{V}_{\bs}$) and boundedness, where in both cases the dominated convergence theorem was used. 
\end{proof}

\begin{theorem}
\Thmconvgdist
\end{theorem}
\begin{proof}

The first statement is trivial as $\tilde{f}_\bs$ has a positive derivative wherever the derivative is defined.
We show the second statement by first verifying that the Robbins-Monro conditions hold almost surely, hence by Corollary~\ref{cor:convg_nu} we have that $\nu_m(\bs)$ converges almost surely to $\tilde{\nu}(\bs)$. As for a constant $H<\infty$ we have $|h_{\ell}| \leq H$  we have that $\sum_{\ell = 1}^m \alpha_\ell \geq \sum_{\ell = 1}^m 1/(\ell H)\rightarrow \infty$ almost surely. 
We verify that $\sum_{\ell = 1}^m \alpha_\ell^2 \leq\infty$ almost surely.
Let $\xi_\ell = h_{\ell}(\nu_\ell(\bs))~- \mathbb{E}[h_{1}(\nu_\ell(\bs))]$ and $\sigma^2_\ell = \mathbb{E}[\xi_\ell^2|\nu_\ell(\bs)].$ 
By the strong law of large numbers for martingales we have (by boundedness of $\xi_\ell$) that almost surely
\begin{equation}
	\lim_{m\rightarrow\infty}\frac{\sum_{\ell = 1}^m\xi_\ell}{\sum_{\ell = 1}^m \sigma^2_\ell}~=~0\label{eq:convgxi}
\end{equation}
Letting $H_\sigma$ be such that $|\sigma^2_\ell|\leq H_\sigma$ for all $\ell$ we can choose $\epsilon\in(0,\inf_\nu \mathbb{E}[h_1(\nu)]/H_\sigma)$ and have an $M\in\mathbb{N}$ such that for all $m> M$
$$\frac{1}{m}\sum_{\ell = 1}^mh_{\ell}(\nu_\ell(\bs))>\inf_\nu \mathbb{E}[h_{1}(\nu)] - \epsilon H_\sigma > 0.$$
We then conclude 
$$\sum_{\ell = 1}^\infty \alpha_\ell^2 \leq \sum_{\ell = 1}^{M} \alpha_\ell^2 + \sum_{\ell = M+1}^{\infty} \frac{1}{m^2(\frac{1}{m}\sum_{\ell = 1}^mh_{\ell}(\nu_\ell(\bs)))^2}< \sum_{\ell = 1}^{M} \alpha_\ell^2 + \sum_{\ell = M+1}^{\infty} \frac{1}{m^2(\inf_\nu \mathbb{E}[h_{1}(\nu)] - \epsilon H_\sigma)^2}<\infty.$$
Hence it follows by Corollary~\ref{cor:convg_nu} that $\nu_{m}(\bs)\stackrel{a.s.}{\rightarrow} \tilde{\nu}(\bs)$.

Now, if $\mathbb{P}(h_{\ell}((\tilde{\nu}(\bs))^-) \neq h_{\ell}((\tilde{\nu}(\bs))^+))>0$ the derivative of $\tilde{f}_\bs$ at $\tilde{\nu}$ would not exist, hence almost surely for all $\ell$ we have  $h_{\ell}((\tilde{\nu}(\bs))^-) = h_{\ell}((\tilde{\nu}(\bs))^+)$
and combined with the above result
it follows that $h_{\ell}(\nu_{\ell}(\bs))\rightarrow h(\tilde{\nu}(\bs))$ (with $h \stackrel{d}{=}h_{1}$ independently)  almost surely by continuous mapping (noting that the set of discontinuity points of $h$ is restricted to a deterministic set). By boundedness we also have
$\mathbb{E}[h_{\ell}(\nu_{\ell}(\bs))]\rightarrow \mathbb{E}[h(\tilde{\nu}(\bs))].$ Similarly, it can be shown that $\sigma^2_\ell \rightarrow \mathbb{E}[\xi^2_1|\tilde{\nu}(\bs)]=\sigma^2$ which is deterministic. Hence $\frac{1}{m}\sum_{\ell = 1}^m \sigma^2_\ell \rightarrow \sigma^2 $ and by \eqref{eq:convgxi} and continuous mapping we have $$m\alpha_m = \frac{m}{|\sum_{\ell = 1}^m h_\ell(\nu_\ell(\bs))|}\rightarrow 1/\mathbb{E}[h_1(\tilde{\nu}(\bs))].$$
The result follows along the lines of \citet{lai1978limit}, with three additional remarks.
\begin{itemize}
	\item In order to use the representation (17) in~\citet{lai1978limit} the recursion of the stochastic approximation has to be truncated earlier, to make sure that the iterates $\nu_\ell(\bs)$ remain in an interval around $\tilde{\nu}(\bs)$ where $\tilde{f}_\bs$ is differentiable, so as to use the mean value theorem.
	\item The martingale central limit theorem (Theorem 2 in~\citet{brown1971martingale}) can be used to show a central limit theorem result for the martingale $\sum_\ell \epsilon_\ell$, as the running mean of the quadratic variation process converges to a constant, we can show a central limit theorem result by just dividing by $\sqrt{m}$.
	\item A supremum law of iterated logarithm for martingales \citet{fisher1986upper} can be used to show that the middle term in (29) in~\citet{lai1979adaptive} goes to zero when divided by $\sqrt{m}.$ 
\end{itemize}

\end{proof}
\section{Notation Table} \label{table:notation}
\FloatBarrier
\begin{table}[h!]
{\footnotesize
	\setlength{\tabcolsep}{0.5mm}	\begin{tabular}{lll}\hline
		Symbol  & Definition & Defined in\\\hline
		\midrule\midrule
		$a$ &  Markov chain/arm & Sec.~\ref{FABP} \\ 
		$t,u$ & Time index  & Sec.~\ref{FABP}\\
		$\bs^a_h,\bs^a,\bs$ & Initial state of an arm & Sec.~\ref{FABP} \\
		$\mathbb{E}_\bs, \mathbb{P}_\bs$ &  Transition kernel and expectation for the Markov chains, conditional on initial state $\bs$& Sec.~\ref{FABP}\\
		$\mathcal{H}$& Set of histories & \eqref{defn:histories}\\
		$R$ & Common reward function for the arms& Sec.~\ref{FABP}\\
		$\pi, \pi^*$& Policy and optimal policy (resp.) for the family of alternative bandit processes& Sec.~\ref{FABP}\\
		$C(\bs^a)$ & Total discounted absolute reward for sampling arm $a$, starting from state $\bs^a$& Sec.~\ref{FABP}\\
		$\mathbb{E}_{\pi}$ &Expectation operator under a fixed policy $\pi$& Sec.~\ref{FABP}\\
		$\gamma$ & Discount factor  & Sec.~\ref{FABP}\\
		$\nu(\bs)$ & Gittins index for state $\bs$& Sec.~\ref{FABP}\\
		$\mathcal{F}_t^a, \mathcal{F}_t$ & Natural filtration generated by $\bS^a$ and $\bS$ (resp.), including starting state & Sec.~\ref{FABP}, \ref{sec:3.1}\\
		$\mathcal{T}^a, \mathcal{T}$ & Set of stopping times w.r.t. $(\mathcal{F}_t^a)_t$ or $(\mathcal{F}_t)_t$ (resp.) & Sec.~\ref{FABP}, \ref{sec:3.1}\\
		$\bS_t^a,\bS_t$ & State of Markov chain/arm $a$ and general arm (resp.) at time $t$& Sec.~\ref{FABP}, \ref{sec:3.1}\\
		$\tau$ & Stopping time, used to determine optimal stopping value& Sec.~\ref{FABP}\\
		$N$ & Sampling horizon for optimal stopping value & Sec.~\ref{sec:3.1} \\
		$\mathcal{T}_N$ & Set of stopping times adapted to $\mathcal{F}_t$ and bounded by $N$& Sec.~\ref{sec:3.1}\\
		$g_t$ & Measurable real-valued cost function& Sec.~\ref{sec:3.1}\\
		$Z_t, Z^{(1)}_t$ & Cost $g_t(\bS_{\ixt})$ of an arm up to time $t$ & Sec.~\ref{sec:3.1}\\
		$Z_t^{(k+1)}$ & $Z_t^{(k)}-\mathbb{E}[\min_{u\in[N]}Z_u^{(k)}\mid \mathcal{F}_t]$& Sec.~\ref{sec:3.1}\\
		$K$& Truncation point of number of nested expectations in optimal stopping approximation  & Sec.~\ref{sec:3.1}\\
		$Z^{(k+1)}_{\bi,t,n}$& Sampling-based approximation of $Z^{(k+1)}_{t}$& Sec.~\ref{sec: 3.2}\\
		$n(\bi,k)$ & Number of simulated paths used to determine $Z^{(k+1)}_{\bi,t,n}$& Sec.~\ref{sec: 3.2}\\
		$V_n^{(K)}, V_\bs(\nu), V_{\bs,m}(\nu)$ & Sampling-based approximation of $\sum_{k=1}^K \mathbb{E}[\min_{u\in[N]}Z_u^{(k)}]$ and  $\sum_{k=1}^K \mathbb{E}_\bs[\min_{u\in[N]}Z_u^{(k)}(\nu)]$& Sec.~\ref{sec: 3.2}, \ref{sect: theoretical_results}\\
		$\sigma$&Stopping time, used to truncate the support of the rewards& \eqref{def:stoppingtime}\\
		$R_u,R_\ell$&Upper and lower bound for reward support (resp.), induced by stopping time $\sigma$& Sec.~\ref{sect: theoretical_results}\\
		$c$& Constant, equal to $(1-\gamma)/(2(R_u-R_\ell)(1-\gamma^N))$&Sec.~\ref{sect: theoretical_results}\\
		$Z_t(\nu)$& Cost $ c\sum_{u=0}^{t\wedge\sigma-1}\gamma^u(\nu-R(\bS_u))$ of an arm up to the minimum of time $t$ and $\sigma$&Sec.~\ref{sect: theoretical_results}\\
		$\nu_{\sigma}(\bs)$& Gittins index approximation found by truncation of horizon and reward support &\eqref{final_approx_GI}\\
		$\nu_m(\bs)$&Stochastic approximation iterates for determining SBGIA& \eqref{stoch_approx}\\
		$\alpha_m$ &Step-size sequence & Sec.~\ref{sect: theoretical_results}\\
		$\nu_M(\bs)$&Sampling-based Gittins index approximation& Sec.~\ref{sect: theoretical_results}\\
		$\tilde{f}_\bs(\nu)$& Expectation of $\mathbb{E}_\bs[V_\bs(\nu)]$&\eqref{def_f}\\
		${f}_\bs(\nu)$& Optimal stopping value $\inf_{\tau\in\mathcal{T}_N}\mathbb{E}_\bs[Z_\tau(\nu)]$&\eqref{def_f}\\
		$B(\delta,\xi)$& Domain error bound, equal to $\xi+\delta$&\eqref{bound_mean_diff}\\
		$\epsilon_m$& Martingale difference sequence & \eqref{MartDiffSeq}\\
		$\bar{\nu}_m,\ubar{\nu}_m$&  Stochastic approximation iterates bounding $\nu_m$ from above and below (resp.) &\eqref{defupperSA}, \eqref{deflowerSA}\\
		$\mathcal{R}_\bs$&Set of roots of $\tilde{f}_\bs$& Sec.~\ref{sec:3.4}\\
		$\tilde{\nu}(\bs)$&Root of $\tilde{f}_\bs$& Sec.~\ref{sec:3.4}\\
		$h_m$&Derivative of $ V_{\bs,m}(\nu)$& Th.~\ref{thm:convg_dist}\\
		$ \Psi_t$ & Sufficient statistics  & Sec.~\ref{ExpFam}\\
		$\kappa_t$&Effective number of observations  & Sec.~\ref{ExpFam}\\
		$O_t^a$& Outcome $t$ for distribution $a$& Sec.~\ref{sec:intro_bayes_experiment}\\
		$\btheta_a$& Parameter for distribution $a$& Sec.~\ref{sec:intro_bayes_experiment}\\
		$\Pi_t^a, \hat{\Pi}_t^a$& Posterior and approximate posterior (resp.) for $\btheta_a$ based on $(O^a_u)_{u=1}^t$ & Sec.~\ref{sec:intro_bayes_experiment}, \eqref{eqn:approx_posterior}\\		
		\hline
\end{tabular}}
\end{table}
\FloatBarrier

% ...





\end{document}